\documentclass[oneside,10pt]{amsart}
\usepackage{amsthm,amssymb}
\usepackage{mathtools}
\usepackage[hidelinks]{hyperref}
\usepackage[a4paper,width=160mm,top=27mm,bottom=27mm]{geometry}
\numberwithin{equation}{section}
\usepackage{todonotes}
\usepackage{color}

\newtheorem{theorem}{Theorem}[section]
\newtheorem{lemma}[theorem]{Lemma}
\newtheorem{definition}[theorem]{Definiton}
\newtheorem{proposition}[theorem]{Proposition}
\newtheorem{corollary}[theorem]{Corollary}

\theoremstyle{definition}
\newtheorem{remx}[theorem]{Remark}
\newenvironment{remark}
  {\pushQED{\qed}\remx}
  {\popQED\endremx}

\begin{document}
\address{Yongming Luo
\newline \indent Faculty of Computational Mathematics and Cybernetics \indent
\newline \indent Shenzhen MSU-BIT University, China\indent
%\newline \indent  1 International University Park Road, Dayun New Town, Longgang District, Shenzhen, Guangdong Province, P.R. China.\indent
}
\email{luo.yongming@smbu.edu.cn}

\newcommand{\diver}{\operatorname{div}}
\newcommand{\lin}{\operatorname{Lin}}
\newcommand{\curl}{\operatorname{curl}}
\newcommand{\ran}{\operatorname{Ran}}
\newcommand{\kernel}{\operatorname{Ker}}
\newcommand{\la}{\langle}
\newcommand{\ra}{\rangle}
\newcommand{\N}{\mathbb{N}}
\newcommand{\R}{\mathbb{R}}
\newcommand{\C}{\mathbb{C}}
\newcommand{\T}{\mathbb{T}}

%%%%%%%%%%%%%%%%%%%%%%%%%%%%%%%%%%%%%%%%
\newcommand{\ld}{\lambda}
\newcommand{\fai}{\varphi}
\newcommand{\0}{0}
\newcommand{\n}{\mathbf{n}}
\newcommand{\uu}{{\boldsymbol{\mathrm{u}}}}
\newcommand{\UU}{{\boldsymbol{\mathrm{U}}}}
\newcommand{\buu}{\bar{{\boldsymbol{\mathrm{u}}}}}
\newcommand{\ten}{\\[4pt]}
\newcommand{\six}{\\[-3pt]}
\newcommand{\nb}{\nonumber}
\newcommand{\hgamma}{H_{\Gamma}^1(\OO)}
\newcommand{\opert}{O_{\varepsilon,h}}
\newcommand{\barx}{\bar{x}}
\newcommand{\barf}{\bar{f}}
\newcommand{\hatf}{\hat{f}}
\newcommand{\xoneeps}{x_1^{\varepsilon}}
\newcommand{\xh}{x_h}
\newcommand{\scaled}{\nabla_{1,h}}
\newcommand{\scaledb}{\widehat{\nabla}_{1,\gamma}}
\newcommand{\vare}{\varepsilon}
\newcommand{\A}{{\bf{A}}}
\newcommand{\RR}{{\bf{R}}}
\newcommand{\B}{{\bf{B}}}
\newcommand{\CC}{{\bf{C}}}
\newcommand{\D}{{\bf{D}}}
\newcommand{\K}{{\bf{K}}}
\newcommand{\oo}{{\bf{o}}}
\newcommand{\id}{{\bf{Id}}}
\newcommand{\E}{\mathcal{E}}
\newcommand{\ii}{\mathcal{I}}
\newcommand{\sym}{\mathrm{sym}}
\newcommand{\lt}{\left}
\newcommand{\rt}{\right}
\newcommand{\ro}{{\bf{r}}}
\newcommand{\so}{{\bf{s}}}
\newcommand{\e}{{\bf{e}}}
\newcommand{\ww}{{\boldsymbol{\mathrm{w}}}}
\newcommand{\zz}{{\boldsymbol{\mathrm{z}}}}
\newcommand{\U}{{\boldsymbol{\mathrm{U}}}}
\newcommand{\G}{{\boldsymbol{\mathrm{G}}}}
\newcommand{\VV}{{\boldsymbol{\mathrm{V}}}}
\newcommand{\II}{{\boldsymbol{\mathrm{I}}}}
\newcommand{\ZZ}{{\boldsymbol{\mathrm{Z}}}}
\newcommand{\hKK}{{{\bf{K}}}}
\newcommand{\f}{{\bf{f}}}
\newcommand{\g}{{\bf{g}}}
\newcommand{\lkk}{{\bf{k}}}
\newcommand{\tkk}{{\tilde{\bf{k}}}}
\newcommand{\W}{{\boldsymbol{\mathrm{W}}}}
\newcommand{\Y}{{\boldsymbol{\mathrm{Y}}}}
\newcommand{\EE}{{\boldsymbol{\mathrm{E}}}}
\newcommand{\F}{{\bf{F}}}
\newcommand{\spacev}{\mathcal{V}}
\newcommand{\spacevg}{\mathcal{V}^{\gamma}(\Omega\times S)}
\newcommand{\spacevb}{\bar{\mathcal{V}}^{\gamma}(\Omega\times S)}
\newcommand{\spaces}{\mathcal{S}}
\newcommand{\spacesg}{\mathcal{S}^{\gamma}(\Omega\times S)}
\newcommand{\spacesb}{\bar{\mathcal{S}}^{\gamma}(\Omega\times S)}
\newcommand{\skews}{H^1_{\barx,\mathrm{skew}}}
\newcommand{\kk}{\mathcal{K}}
\newcommand{\OO}{O}
\newcommand{\bhe}{{\bf{B}}_{\vare,h}}
\newcommand{\pp}{{\mathbb{P}}}
\newcommand{\ff}{{\mathcal{F}}}
\newcommand{\mWk}{{\mathcal{W}}^{k,2}(\Omega)}
\newcommand{\mWa}{{\mathcal{W}}^{1,2}(\Omega)}
\newcommand{\mWb}{{\mathcal{W}}^{2,2}(\Omega)}
\newcommand{\twos}{\xrightharpoonup{2}}
\newcommand{\twoss}{\xrightarrow{2}}
\newcommand{\bw}{\bar{w}}
\newcommand{\bz}{\bar{{\bf{z}}}}
\newcommand{\tw}{{W}}
\newcommand{\tr}{{{\bf{R}}}}
\newcommand{\tz}{{{\bf{Z}}}}
\newcommand{\lo}{{{\bf{o}}}}
\newcommand{\hoo}{H^1_{00}(0,L)}
\newcommand{\ho}{H^1_{0}(0,L)}
\newcommand{\hotwo}{H^1_{0}(0,L;\R^2)}
\newcommand{\hooo}{H^1_{00}(0,L;\R^2)}
\newcommand{\hhooo}{H^1_{00}(0,1;\R^2)}
\newcommand{\dsp}{d_{S}^{\bot}(\barx)}
\newcommand{\LB}{{\bf{\Lambda}}}
\newcommand{\LL}{\mathbb{L}}
\newcommand{\mL}{\mathcal{L}}
\newcommand{\mhL}{\widehat{\mathcal{L}}}
\newcommand{\loc}{\mathrm{loc}}
\newcommand{\tqq}{\mathcal{Q}^{*}}
\newcommand{\tii}{\mathcal{I}^{*}}
\newcommand{\Mts}{\mathbb{M}}
\newcommand{\pot}{\mathrm{pot}}
\newcommand{\tU}{{\widehat{\bf{U}}}}
\newcommand{\tVV}{{\widehat{\bf{V}}}}
\newcommand{\pt}{\partial}
\newcommand{\bg}{\Big}
\newcommand{\hA}{\widehat{{\bf{A}}}}
\newcommand{\hB}{\widehat{{\bf{B}}}}
\newcommand{\hCC}{\widehat{{\bf{C}}}}
\newcommand{\hD}{\widehat{{\bf{D}}}}
\newcommand{\fder}{\partial^{\mathrm{MD}}}
\newcommand{\Var}{\mathrm{Var}}
\newcommand{\pta}{\partial^{0\bot}}
\newcommand{\ptaj}{(\partial^{0\bot})^*}
\newcommand{\ptb}{\partial^{1\bot}}
\newcommand{\ptbj}{(\partial^{1\bot})^*}
\newcommand{\geg}{\Lambda_\vare}
\newcommand{\tpta}{\tilde{\partial}^{0\bot}}
\newcommand{\tptb}{\tilde{\partial}^{1\bot}}
\newcommand{\ua}{u_\alpha}
\newcommand{\pa}{p\alpha}
\newcommand{\qa}{q(1-\alpha)}
\newcommand{\Qa}{Q_\alpha}
\newcommand{\Qb}{Q_\eta}
\newcommand{\ga}{\gamma_\alpha}
\newcommand{\gb}{\gamma_\eta}
\newcommand{\ta}{\theta_\alpha}
\newcommand{\tb}{\theta_\eta}

%%%%%%%%%%%%%%%%%%%%%%%%%%

\newcommand{\mH}{{E}}
\newcommand{\mN}{{N}}
\newcommand{\mD}{{\mathcal{D}}}
\newcommand{\csob}{\mathcal{S}}
\newcommand{\mA}{{A}}
\newcommand{\mK}{{Q}}
\newcommand{\mS}{{S}}
\newcommand{\mI}{{I}}
\newcommand{\tas}{{2_*}}
\newcommand{\tbs}{{2^*}}
\newcommand{\tm}{{\tilde{m}}}
\newcommand{\tdu}{{\phi}}
\newcommand{\tpsi}{{\tilde{\psi}}}
\newcommand{\Z}{{\mathbb{Z}}}
\newcommand{\tsigma}{{\tilde{\sigma}}}
\newcommand{\tg}{{\tilde{g}}}
\newcommand{\tG}{{\tilde{G}}}
\newcommand{\mM}{{M}}
\newcommand{\mC}{\mathcal{C}}
\newcommand{\wlim}{{\text{w-lim}}\,}
\newcommand{\diag}{L_t^\ba L_x^\br}
\newcommand{\vu}{ u}
\newcommand{\vz}{ z}
\newcommand{\vv}{ v}
\newcommand{\ve}{ e}
\newcommand{\vw}{ w}
\newcommand{\vf}{ f}
\newcommand{\vh}{ h}
\newcommand{\vp}{ \vec P}
\newcommand{\ang}{{\not\negmedspace\nabla}}
\newcommand{\dxy}{\Delta_{x,y}}
\newcommand{\lxy}{L_{x,y}}
\newcommand{\gnsand}{\mathrm{C}_{\mathrm{GN},3d}}
\newcommand{\wmM}{\widehat{{M}}}
\newcommand{\wmH}{\widehat{{E}}}
\newcommand{\wmI}{\widehat{{I}}}
\newcommand{\wmK}{\widehat{{Q}}}
\newcommand{\wmN}{\widehat{{N}}}
\newcommand{\wm}{\widehat{m}}
\newcommand{\ba}{\mathbf{a}}
\newcommand{\bb}{\mathbf{b}}
\newcommand{\br}{\mathbf{r}}
\newcommand{\bs}{\mathbf{s}}
\newcommand{\bq}{\mathbf{q}}
\newcommand{\SSS}{\mathcal{S}}
\newcommand{\re}{\mathrm{Re}}
\newcommand{\im}{\mathrm{Im}}
\newcommand{\zt}{{\tilde z}}
%-------------------------------------

%New Commands

%-------------------------------------

%%%%%%%%%%%%%%%%%%%%%%
\title[Almost sure scattering for cubic NLS on $\R^3\times\T$]
{Almost sure scattering for the defocusing cubic nonlinear Schr\"odinger equation on $\R^3\times\T$}
\author{Yongming Luo}

%\date{}
\maketitle

\begin{abstract}
We consider the Cauchy problem for the defocusing cubic nonlinear Schr\"odinger equation (NLS) on the waveguide manifold $\R^3\times\T$ and establish almost sure scattering for random initial data, where no symmetry conditions are imposed and the result is available for arbitrarily rough data $f\in H^s$ with $s\in\R$. The main new ingredient is a layer-by-layer refinement of the newly established randomization introduced by Shen-Soffer-Wu \cite{ShenSofferWu21}, which enables us to also obtain strongly smoothing effect from the randomization for the forcing term along the periodic direction. It is worth noting that such smoothing effect generally can not hold for purely compact manifolds, which is on the contrary available for the present model thanks to the mixed type nature of the underlying domain. As a byproduct, by assuming that the initial data are periodically trivial, we also obtain the almost sure scattering for the defocusing cubic NLS on $\R^3$ which parallels the ones by Camps \cite{Camps} and Shen-Soffer-Wu \cite{Shen2022}. To our knowledge, the paper also gives the first almost sure well-posedness result for NLS on product spaces.
\end{abstract}

\section{Introduction}
In this paper, we study the Cauchy problem for the defocusing cubic nonlinear Schr\"odinger equation (NLS)
\begin{align}\label{nls}
(i\pt_t+\Delta_{x,y})u=|u|^2u
\end{align}
on the semiperiodic space $\R_x^3\times\T_y$. Equation \eqref{nls}, or more generally the NLS-models on $\R^d\times\T^n$, are usually referred to as the NLS on waveguide manifolds which play a fundamental role in the study of nonlinear optics (see for instance \cite{waveguide_ref_1,waveguide_ref_2,waveguide_ref_3}).
From a mathematical point of view, the mixed type nature of $\R^d\times\T^n$ also leads to various new challenges which can not be solved by using the methods developed for the purely Euclidean case $\R^d$. For example, due to the rather weak dispersion along the torus side, deriving the Strichartz estimates on $\T^n$ is indeed a very challenging task where many advanced theories, such as the number theoretical and $\ell^2$-decoupling methods, are needed (see e.g. \cite{Bourgain2,l2decoupling}).

On the other hand, the mixed nature of $\R^d\times \T^n$ also leads to the following interesting question: While the periodic NLS-wave in general does not scatter, its Euclidean counterpart will indeed become asymptotically linear under certain circumstances. It turns out that scattering solutions are still obtainable on waveguide manifolds. There have been nowadays extensive references for the study of the well-posedness and long time behavior problems of the NLS on waveguide manifolds. In this direction, we refer for instance to \cite{TNCommPDE,TzvetkovVisciglia2016,TTVproduct2014,Ionescu2,HaniPausader,CubicR2T1Scattering,R1T1Scattering,Cheng_JMAA,ZhaoZheng2021,RmT1,YuYueZhao2021,Luo_Waveguide_MassCritical,Luo_inter,Luo_energy_crit,luo2022sharp}. We point out that in view of scaling arguments, the model \eqref{nls} is $H^{\frac12}$-critical w.r.t $\R^3$ and energy-critical w.r.t. $\R^3\times\T$, thus it is considered as a model where scattering shall take place (see related discussions in \cite{HaniPausader}). In fact, by appealing to the concentration compactness principle initiated by Kenig and Merle \cite{KenigMerle2006}, it was shown by Zhao \cite{RmT1} that \eqref{nls} always possesses a global scattering solution for arbitrary initial data lying in the energy space. It is also worth mentioning that a corresponding large data scattering result for the focusing analogue of \eqref{nls} was recently established by the author of the present paper by appealing to the so-called semivirial vanishing theory, see \cite{Luo_energy_crit}.

Notice that in the above mentioned references, the well-posedness results are established for initial data lying in the critical or subcritical spaces. On the other hand, since the seminal work of Christ, Colliander and Tao \cite{ill_posed} we know that the NLS-problems are generally ill-posed for supercritical data. Nevertheless, by making use of the probabilistic tools, Bourgain \cite{BourgainProb1,BourgainProb2} was able to prove that those ``badly-behaved'' solutions are only the exceptional ones and in general the well-posedness of the periodic NLS is also expected for low regular initial data coming from a probability set with full measure. On the other hand, the results from \cite{BourgainProb1,BourgainProb2} are closely linked to the invariant Gibbs measure and hence only applicable for spaces with very low regularity (in 2D and higher dimensional spaces, the solutions are even merely distributions). To overcome such difficulties, it was a crucial observation by Burq and Tzvetkov \cite{BurqTzvetkov1,BurqTzvetkov2} that the randomization will indeed provide certain smoothing effect for raising the integrability of the \textit{a priori} estimates. Utilizing this key finding, Burq and Tzvetkov were able to establish local and global well-posedness results for the supercritical wave equation on compact manifolds with random data.

We underline, however, that the results from \cite{BourgainProb1,BourgainProb2,BurqTzvetkov1,BurqTzvetkov2} are closely related to the compact structure of the underlying domain, hence the so far developed methods can not be directly used to deduce similar results for problems posed on domains with infinite size. Such difficulty can be overcome by associating the given supercritical data with another type of randomization, such as the so-called \textit{Wiener randomization} based on a unit-scale decomposition of the frequency space, see e.g. \cite{benyi_oh_pocov_survey} for related discussions.
%See e.g. \cite{benyi_oh_pocov_survey} for a more detailed survey of the related studies in this direction.
It is also worth noting that despite the randomization generally does not improve the differentiability of the initial data, it indeed improves their integrability which is compatible with the Lebesgue-type norms deduced from the Strichartz estimates. This is the main reason why we are able to obtain low regularity well-posedness results by randomizing the data.

In this paper, we aim to establish some first almost sure well-posedness and scattering results for the NLS-models posed on the waveguide manifolds. To our knowledge, such problems have so far not been considered in the existing references. As we shall see, such new results can not be seen as a simple and straightforward extension from the ones for the Euclidean case. To be more precise, the main issue here is that the underlying domain is partially periodic, in which case the scattering (at least along the periodic direction) is not expected to happen. The main novelty of the paper is the construction of a suitable randomization procedure, relying on which we are still able to benefit certain decay for the \textit{a priori} estimates from the strong dispersion of the NLS-wave on $\R^3$, which in turn enables us to ultimately obtain the scattering on the whole space $\R^3\times\T$. Particularly, our result requires no symmetry conditions and is available for any rough data $f\in H^s$ with $s\in\R$.

Before we turn to the precise randomization of the initial data, we shall still review some well-known almost sure well-posedness and scattering results which are mostly related to our study. Among all, we first underline that because of its simple form and wide applicability, the Wiener randomization has become one of the mostly applied randomization method nowadays for studying the almost sure well-posedness and scattering problems of the NLS, see e.g. \cite{Benyi2015,Benyi2019,Brereton} and the references therein.

Nevertheless, by invoking solely the randomization methods one usually ends up with local, small data or conditional global results. To deduce unconditional global results, monotonicity formulas such as the Morawetz inequalities and energy increments will usually come into play. For the energy-critical nonlinear wave equation (NLW), the first almost sure global well-posedness result was established by Pocovnicu \cite{PocovnicuNLW2017} on $\R^4$ and $\R^5$. A similar result was later extended to $\R^3$ by Oh and Pocovnicu \cite{OhPocovnicu2016NLW}. On the other hand, the first almost sure scattering result for the energy-critical NLW on $\R^4$ with radial random data in $H_{\rm rad}^{s}(\R^4)$, $s\in(\frac12,1)$, was given by Dodson, L\"uhrmann and Mendelson \cite{DLM20}.

Unlike the NLW-case, new difficulties arise in the study for the NLS-problems since the energy of the NLS does not control the term $\|\pt_t u(t)\|^2_{L^2(\R^d)}$. By also invoking the additional decay gained from the radial symmetry assumption, Killip, Murphy, and Visan \cite{KMV4dprob} were able to prove the almost sure scattering in $H^s_{\rm rad}(\R^4)$, $s\in(\frac56,1)$, for the cubic NLS on $\R^4$. The result was later improved by Dodson, Lührmann and Mendelson \cite{DLM19} to the range $s\in(\frac12,1)$. By also using certain high-low frequency decomposition technique, Camps \cite{Camps} and Shen-Soffer-Wu \cite{Shen2022} have independently proved the almost sure scattering for the cubic NLS on $\R^3$ with radial random data. It is also worth mentioning the recent almost sure global well-posedness result by Oh, Okamoto, and Pocovnicu \cite{GWP5d6d} for the energy-critical NLS on $\R^5$ and $\R^6$, where no radial assumption was needed.

Lastly, we also review the almost sure well-posedness results which make use of randomization other than the Wiener randomization. The main purpose of using non-standard randomizations lies in the fact that they usually raise much stronger smoothing effect than the one provided by the Wiener randomization, while the construction of such randomizations generally requires a much higher cost. In \cite{Bringmann2021NLW} Bringmann exploited a unit-scale decomposition in both physical and frequency space, in order to establish the almost sure scattering for the non-radial NLW in 4D. By using another randomization based on the annuli decomposition, Bringmann \cite{BringmannRadNLW2020} was able to prove the almost sure scattering for the quintic NLW in $H^s_{\rm rad}(\R^3)$ for any $s>0$. By appealing to a randomization where the angular variable was also taken into account, Spitz \cite{SpitzNLS,SpitzNLW} was able to give the almost sure scattering results for both the non-radial 4D cubic NLS and NLW. Finally, we also mention the very recent work \cite{ShenSofferWu21} by Shen, Soffer and Wu, where the authors proved the almost sure scattering for the 3D and 4D energy-critical NLS. Surprisingly, their result is available in the non-radial case and even for any rough data $f\in H^s(\R^d)$ with $s\in\R$.

In this paper, we follow closely the strategy in \cite{ShenSofferWu21} to prove the almost sure scattering result for \eqref{nls}. We begin with the construction of the randomization, which is detailed in the following section.

\subsection{Randomization}\label{sec random}
Let $K\in\N$ and $N\in 2^{\N_0}$ be a dyadic number. Define
\begin{gather*}
O_N:=\{\xi\in\R^d:|\xi_j|\leq N\quad\forall\,j=1,\cdots d\},\quad
Q_N:=O_{2N}\setminus O_N.
\end{gather*}
Write $Q_1:=O_1$ and for $N\geq 2$ define
\begin{align*}
\mathcal{A}(Q_N):=\{Q : \text{ $Q$ is a dyadic cube with length $N^{-K}$ and $Q\subset Q_N$}\}.
\end{align*}
We first partition $\R^d$ according to $\mathcal{A}(O_N)$:
\begin{align*}
\R^d=Q_1\cup(\cup_{N\in 2^{\N}}\cup_{Q\in\mathcal{A}(Q_N)} Q).
\end{align*}
For educational purpose, if we terminate our decomposition procedure at this point and assign the randomization to the derived decomposition:
\[f^\omega=\sum_j g_j\Box_j f\]
(we will give the precise meaning of the notation right after the educational example), then we arrive at the randomization in \cite{ShenSofferWu21}. A naive idea to extend from the $\R^3$-case studied in \cite{ShenSofferWu21} to $\R^3\times\T$ would be the following: By writing a function $f(x,y)=\sum_{k}f_k(x)e^{iky}$ we define its randomization as
\[f^\omega(x,y):=\sum_k\sum_j g_j\Box_j f_k(x)e^{iky}.\]
Nevertheless, such randomization is rather useless for our purpose. Roughly speaking, by defining the randomization in such a way, the smoothing effect provided by the multiplier $\Box_j$ will be completely ignored by the Fourier modes located in the region $\{|k|\gg|\xi|\}$. To overcome such difficulties, we then make a finer decomposition at each discrete Fourier mode $e^{iky}$ according to its frequency size.

Our construction is as follows: For any $k\in\Z^n\setminus \{0\}$ we can find some $M\in 2^{\N}$ such that $2^{-1}M\leq |k|< M$. For $N\leq M$, we partition $Q_N$ into cubes of length $M^{-K}$:
\begin{align*}
\mathcal{B}_M(Q_N):=\{Q : \text{ $Q$ is a dyadic cube with length $M^{-K}$ and $Q\subset Q_N$}\}.
\end{align*}
Then we partition $\R^d\times\Z^n$ into
\begin{align*}
\R^d\times\Z^n&=(Q_1,0)\cup\bg(\cup_{N\in 2^{\N}}\cup_{Q\in\mathcal{A}(O_N)}( Q,0)\bg)\nonumber\\
&\cup_{M\geq 1}\cup_{k\in\Z^n:M\leq |k|< 2M}\bg[\big(\cup_{N\leq M}\cup_{Q\in\mathcal{B}_M(Q_N)}(Q,k)\big)\nonumber\\
&\cup\big(\cup_{N>M}\cup_{Q\in\mathcal{A}(Q_N)}(Q,k)\big)\bg].
\end{align*}
A schematic description of the dyadic decomposition is found in Fig. \ref{decomp} below.
\begin{figure}[ht!]
  % Requires \usepackage{graphicx}
  \centering
  \includegraphics[width=70mm]{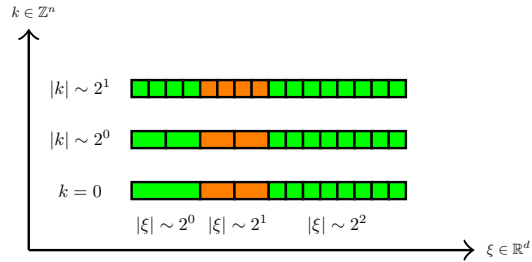}
  \caption{A schematic description of the dyadic decomposition for $K=1$.}\label{decomp}
\end{figure}

By re-enumeration we denote by $\mathcal{Q}$ the collection of the above deduced dyadic cubes:
\begin{align*}
\mathcal{Q}=\{Q_{jk}\subset\R^d,\,j=j_{\R^d}\in\N,\,k=k_{\Z^n}\in\Z^n\}.
\end{align*}
Next, let $\tilde{\psi}_{jk}\in C_0^\infty(\R^d;[0,1])$ be a cut-off function such that
\begin{align*}
\tilde{\psi}_{jk}(\xi)=\left\{
\begin{array}{ll}
1,&\xi\in Q_{jk},\\
0,&\xi\in 2Q_{jk},
\end{array}
\right.
\end{align*}
where 2$Q_{jk}$ is the cube with the same center as $Q_{jk}$ with $\mathrm{diam}\,(2Q_{jk} )=2\mathrm{diam}\,(Q_{jk} )$. Define now $\psi_{jk}$ by
\begin{align*}
\psi_{jk}(\xi)=\frac{\tilde{\psi}_{jk}(\xi)}{\sum_{j'\in\N}\tilde{\psi}_{j'k}(\xi)}.
\end{align*}
For $f:\R^d\to\C$, we define
\begin{align*}
\Box_{jk}f(x):=\mathcal{F}_x^{-1}(\psi_{jk}(\xi)\mathcal{F}_xf(\xi))(x),
\end{align*}
where $\mathcal{F}_x$ denotes the continuous $x$-directional Fourier transform. Let $(g_{jk})_{j\in\N,k\in\Z^n}$ be a sequence of i.i.d complex Gaussians defined on a probability space $(\Omega,\mathcal{F},\mathbb{P})$. By writing $f(x,y)=\sum_{k\in\Z^n}f_k(x)e^{iky}$ we finally define our randomization $f^\omega(x,y)$ as
\begin{align}\label{1.7}
f^{\omega}(x,y):=\sum_{j\in\N,k\in\Z^n}g_{jk}(\omega)\Box_{jk}f_k(x)e^{iky}
\end{align}
for a given sample $\omega\in\Omega$.

\subsection{Main result}
Having defined the randomization we are now able to state our main result.
\begin{theorem}\label{main thm}
Let $s\in\R$ and $f\in H^s(\R^3\times\T)$. Suppose that the number $K\in\N$ defined in Section \ref{sec random} satisfies
\begin{gather}\label{condition on K}
K>\max\left\{-4s+\frac{20}{3},\,-6s+\frac23,\,\frac83\right\}.
%-2s+\frac{14}{3},\,
\end{gather}
For a given sample $\omega\in\Omega$ we define the function $f^\omega$ through \eqref{1.7} associated to $f$. Then for a.e. $\omega\in\Omega$ the NLS \eqref{nls} possesses a global solution $u$ satisfying
\[u-e^{it\Delta}f^\omega\in C(\R;H^1(\R^3\times\T)).\]
Moreover, the solution $u$ scatters in the sense that there exist $\phi^\pm\in H^1(\R^3\times\T)$ such that
\[\lim_{t\to\infty}\|u-e^{it\Delta}f^\omega-e^{it\Delta}\phi^\pm\|_{H^1(\R^3\times\T)}=0.\]
\end{theorem}

\begin{remark}
By assuming that $f\in H^s(\R^3\times\T)$ is periodically trivial, Theorem \ref{main thm} also gives the almost sure scattering for the cubic NLS on $\R^3$ which parallels the ones established earlier by Camps \cite{Camps} and Shen-Soffer-Wu \cite{Shen2022}. In particular, we assume no symmetry assumptions and the result is available for arbitrarily rough data.
\end{remark}

We basically follow the same line as in \cite{ShenSofferWu21} to prove Theorem \ref{main thm}, which can be briefly summarized as follows:
\begin{itemize}
\item[(i)] We first prove the probabilistic Strichartz estimates adapted to the randomization defined in Section \ref{sec random} (Lemma \ref{prob strichartz lem}). As we shall see, such probabilistic Strichartz estimates will exhibit a supercritical scaling, which turn out to be very useful for H\"older-type estimates.

\item[(ii)] We then prove a local well-posedness result for the solution $w$ of the perturbed NLS \eqref{nls_w} (Lemma \ref{local theo lem}), which serves as the starting point of our bootstrap proof cycle.

\item[(iii)]Afterwards, we prove an interaction Morawetz inequality for the solution $w$ of \eqref{nls_w} (Lemma \ref{inter mora lem}). Followed by a double bootstrap argument we shall finally arrive at the almost conservation laws for $w$ (Lemma \ref{almost conv lem}) which provide some uniform (in time) bounds for the solution $w$.

\item[(iv)]Finally, we partition the time line $\R$ into small pieces where the $\|v\|_W$-norm on each small subinterval is small. By using a suitable stability result (Lemma \ref{long time stab lem}) we are able to infer the $X^1$-finiteness of $w$ on the first small interval. Thanks to the uniform boundedness properties of $w$ deduced from Lemma \ref{almost conv lem}, such perturbation arguments are indeed applicable for all given small subintervals by also combining with a standard inductive argument. The desired claim then follows from summing up the estimates on the small pieces.
\end{itemize}

It is however worth pointing out that albeit the proof routine is very similar to the one from \cite{ShenSofferWu21}, some new technical difficulties do indeed arise in the waveguide setting. As a major technical difficulty, we shall establish our \textit{a priori} estimates based on the framework of atomic spaces. Notice that similar estimates in the deterministic setting were already given in \cite{RmT1}, which unfortunately made use of the $X^1$-norm everywhere in the proofs and thus are not totally compatible with the random setting. For this purpose, we also need to prove several new bilinear and nonlinear estimates for the forcing term $v=e^{it\Delta}v_0$ and the nonlinear term $vw^2$ respectively. On the other hand, we shall also invoke certain Sobolev embedding on $\T$ to modify the interaction Morawetz inequality in the waveguide setting which also leads to some additional lengthy calculations. For more details, we refer to Section \ref{sec 3.2} and Section \ref{sec 4.2} respectively.

The paper is organized as follows: In Section \ref{bigsec2} we collect some preliminary tools which will be used throughout the paper. In Section \ref{bigsec3} we prove some useful estimates. In Section \ref{bigsec4} we give the proof of Theorem \ref{main thm}.

\subsection{Notation and definitions}

Throughout the paper, we ignore in most cases the dependence of the function spaces on their underlying domains and hide this dependence in their indices. For example $L_x^2=L^2(\R^3)$, $H_{x,y}^1= H^1(\R^3\times \T)$
and so on. However, when the space is involved with time, we still display the underlying temporal interval such as $L_t^pL_x^q(I)$, $L_t^\infty L_{x,y}^2(\R)$ etc. For a number $a\in\R$, we denote by $a\pm:=a\pm\vare$ for arbitrary $\vare>0$. We define the $W$-norm by
\begin{align}\label{w norm}
\|v\|_{W}&:=\|\la\nabla\ra^{(s-\frac13+\frac{K}{2})-}v\|_{L_t^2 L_{x,y}^\infty}+\|v\|_{L_{t,x,y}^4}+\|v\|_{L_t^6 L_{x,y}^3}.
\end{align}
We use $\mathcal{F}(f)$ or $\hat{f}$ to denote the $x$-directional Fourier transform, i.e.
\[\mathcal{F}(f)(\xi)=\hat{f}(\xi)=\int_{\R^3}f(x)e^{i\xi\cdot x}\,dx.\]
As usual, we use $P_N$, $P_{\leq N}$, $P_{>N}$ to denote the Littlewood-Paley projectors for a dyadic number $N\geq 1$, where we also make the convention that $P_1=P_{\leq 1}$. Finally, for $s\in[0,\frac{d}{2})$ we say that a pair $(p,q)$ is $H^s$-admissible if $2\leq p,q\leq \infty$, $(p,q,d,s)\neq (2,\infty,2,0)$ and $\frac{2}{p}+\frac{d}{q}=\frac{d}{2}-s$.
\section{Some preliminaries}\label{bigsec2}
In this section we collect some auxiliary tools which will be useful for the upcoming proof. Though we are considering a problem on $\R^3\times\T$, the results in this section will mostly be formulated for the general domain $\R^d\times \T^n$, which might be of independent interest.

\subsection{Isometries and inequalities}
\begin{lemma}[Orthogonality]\label{lemma 3.1}
For $f\in L^2_{x,y}$ with $f(x,y)=\sum_k f_k(x)e^{iky}$ we have
\begin{align*}
\|\Box_{jk}f_k(x)e^{iky}\|_{L_{x,y}^2\ell^2_{jk}}=\|\Box_{jk}f_k\|_{L_{x}^2\ell^2_{jk}}\sim \|f\|_{L_{x,y}^2}.
\end{align*}
\end{lemma}

\begin{proof}
The proof follows the same line of \cite[Lem. 2.6]{ShenSofferWu21}. By our construction of the decomposition, for the set
\begin{align*}
\mathcal{B}_{jk}:=\{j'\in \N:\mathrm{supp}\,\psi_{j'k}\cap \mathrm{supp}\,\psi_{jk}\neq \varnothing\}
\end{align*}
we have $\sup_{jk}|\mathcal{B}_{jk}|<\infty$. For each $k\in\Z^n$ we also have
\begin{align*}
j\in \mathcal{B}_{j'k}\Leftrightarrow j'\in \mathcal{B}_{jk}.
\end{align*}
Using Plancherel and the fact that $\sum_j\psi_{jk}=1$ holds for any $k\in\Z^n$ we obtain
\begin{align*}
\|f\|_{L_{x,y}^2}^2=\sum_k\sum_{j,j':j'\in \mathcal{B}_{j}}\int_{\R^d}\psi_{jk}(\xi)\psi_{j'k}(\xi)|\hat{f}_k(\xi)|^2\,d\xi.
\end{align*}
Hence
\begin{align*}
\|\Box_{jk}f_k(x)e^{iky}\|^2_{L_{x,y}^2\ell^2_{jk}}&=\sum_{jk}\int_{\R^d}|\psi_{jk}\hat{f}_k|^2\,d\xi\nonumber\\
&\leq \sum_k\sum_{j,j':j'\in \mathcal{B}_{j}}\int_{\R^d}\psi_{jk}(\xi)\psi_{j'k}(\xi)|\hat{f}_k(\xi)|^2\,d\xi=\|f\|_{L_{x,y}^2}^2.
\end{align*}
On the other hand, using Cauchy-Schwarz
\begin{align*}
\|f\|_{L_{x,y}^2}^2&=\sum_k\sum_{j,j':j'\in \mathcal{B}_j}\int_{\R^d}\psi_{jk}\psi_{j'k}|\hat{f}_k|^2\,d\xi
\lesssim \sum_k\sum_{j,j':j'\in \mathcal{B}_j}\int_{\R^d}(\psi_{jk}^2+\psi_{j'k}^2)|\hat{f}_k|^2\,d\xi\nonumber\\
&\lesssim \sum_k\sum_{j,j':j'\in \mathcal{B}_j}\int_{\R^d}\psi_{jk}^2|\hat{f}_k|^2\,d\xi\lesssim
\sum_{jk}\int_{\R^d}|\psi_{jk}\hat{f}_k|^2\,d\xi=\|\Box_{jk}f_k(x)e^{iky}\|^2_{L_x^2\ell^2_{jk}},
\end{align*}
as desired.
\end{proof}

\begin{lemma}[$L^q$-$L^p$ estimate]\label{lemma 3.2}
Let $2\leq p\leq q\leq\infty$. Then for any $j\in\N$ and $k\in\Z^n$ we have
\begin{align}
\|\Box_{jk}f(x)e^{iky}\|_{L_{x,y}^q}\lesssim \|\la\nabla\ra^{-K(\frac{d}{p}-\frac{d}{q})}\Box_{jk}f(x)e^{iky}\|_{L_{x,y}^p}.\label{3.7}
\end{align}
\end{lemma}

\begin{proof}
The proof is based on a modification of \cite[Lem. 2.7]{ShenSofferWu21} adapted to our randomization defined in Section \ref{sec random}. Using the property of the decomposition $\mathcal{Q}$, we can find some $N\in 2^{\N_0}$ such that $\mathrm{supp}\,\psi_{jk}\subset\{|\xi|\sim N\}$. If $N\gtrsim |k|$, then by our construction of the decomposition we have $|\mathrm{supp}\,\psi_{jk}|\sim N^{-Kd}$ and
\begin{gather*}
\|\Box_{jk}f(x)e^{iky}\|_{L_{x,y}^p}\lesssim N^{K(\frac{d}{p}-\frac{d}{q})}\|\la \nabla \ra^{-K(\frac{d}{p}-\frac{d}{q})}\Box_{jk}f(x)e^{iky}\|_{L_{x,y}^p}.
\end{gather*}
Let now $r\geq 2$ and $\theta\in[0,1]$ such that $(q')^{-1}=r^{-1}+2^{-1}$ and $p^{-1}=\theta 2^{-1}+(1-\theta)q^{-1}$. Using Hausdorff-Young and H\"older we deduce that
\begin{align*}
\|\Box_{jk}f(x)e^{iky}\|_{L_{x,y}^q}&\leq \|\psi_{jk}\hat{f} \delta_k(k')\|_{L_\xi^{q'}\ell_{k'}^{q'}}=
\|\psi_{jk}\hat{f} \|_{L_\xi^{q'}}\nonumber\\
&\leq \|\chi_{\xi \in 2Q_{jk}}\|_{L_\xi^{r}}\|\psi_{jk}\hat{f} \|_{L_\xi^{2}}\lesssim N^{-\frac{Kd}{r}}\|\Box_{jk}f(x)e^{iky}\|_{L_{x,y}^2}.
\end{align*}
Notice also that $\theta r^{-1}=p^{-1}-q^{-1}$. Hence interpolating with the trivial identity $\|\Box_{jk}f(x)e^{iky}\|_{L_{x,y}^q}=\|\Box_{jk}f(x)e^{iky}\|_{L_{x,y}^q}$ yields
\begin{align}
\|\Box_{jk}f(x)e^{iky}\|_{L_{x,y}^q}&\lesssim N^{-\frac{\theta Kd}{r}}\|\Box_{jk}f(x)e^{iky}\|_{L_{x,y}^p}
=N^{-K(\frac{d}{p}-\frac{d}{q})}\|\Box_{jk}f(x)e^{iky}\|_{L_{x,y}^p}\nonumber\\
&\lesssim \|\la \nabla \ra^{-K(\frac{d}{p}-\frac{d}{q})}\Box_{jk}f(x)e^{iky}\|_{L_{x,y}^p}.\label{3.8}
\end{align}
In the case $N\lesssim |k|$ we will instead have $|\mathrm{supp}\,\psi_{jk}|\sim |k|^{-Kd}$ and
\begin{gather*}
\|\Box_{jk}f(x)e^{iky}\|_{L_{x,y}^p}\lesssim |k|^{K(\frac{d}{p}-\frac{d}{q})}\|\la \nabla \ra^{-K(\frac{d}{p}-\frac{d}{q})}\Box_{jk}f(x)e^{iky}\|_{L_{x,y}^p},
\end{gather*}
from which \eqref{3.8} follows in a similar way. This completes the proof.
\end{proof}

\subsection{Probabilistic preliminaries}
In the following we collect some useful tools from the probability theory. Together with the probabilistic Strichartz estimates (Lemma \ref{prob strichartz lem}) we will then establish some uniform bounds for the forcing term $v=e^{it\Delta}f^\omega$ which play a crucial role in the proof of Theorem \ref{main thm}.

\begin{lemma}[Large deviation estimate, \cite{BurqTzvetkov1}]\label{lemma 6.1}
Let $(g_n)_n$ be a sequence of i.i.d. Gaussians. Then there exists some $C>0$ such that for all $(c_n)_n\in \ell^2_n$, $\ld>0$ and $2\leq p<\infty$ we have
\begin{gather*}
\mathbb{P}\bg(\{\omega\in\Omega:|\sum_{n\in\N}c_n g_n(\omega)|>\ld\}\bg)\leq 2\exp\{-C\ld\|c_n\|_{\ell_n^2}^{-2}\},\\
\|\sum_{n\in\N}c_n g_n(\omega)\|_{L_\omega^p}\leq C\sqrt{p}\|c_n\|_{\ell_n^2}.
\end{gather*}
\end{lemma}

\begin{lemma}[Almost sure finiteness, \cite{Tzvetkov10,DLM19,DLM20}]\label{2.4 lem}
Let $F$ be a measurable function and suppose that there exist $C_0,K > 0$ and $p_0 \geq 1$ such that for any $p \geq p_0$ we have
\[\|F\|_{L_\omega^p}\leq C_0\sqrt{p}K.\]
Then there exist $c,C_1>0$, depending on $C_0$ and $p_0$ but not on $K$, such that for any $\ld>0$ we have
\[
\mathbb{P}\bg(\{\omega\in\Omega:|F(\omega)|>\ld\}\bg)\leq C_1\exp\{-c\ld^2 K^{-2}\}.
\]
In particular,
\[
\mathbb{P}\bg(\{\omega\in\Omega:|F(\omega)|<\infty\}\bg)=1.
\]
\end{lemma}

\subsection{Function spaces}
Next, we define the function spaces and collect some of their useful properties which will be used for the Cauchy problem \eqref{nls}, \eqref{nls_w} and \eqref{nls_g1}. We begin with the definitions of $U^p$- and $V^p$-spaces introduced in \cite{HadacHerrKoch2009}.

\begin{definition}[$U^p$-spaces]\label{def up}
Let $1\leq p < \infty$, $\mathcal{H}$ be a complex Hilbert space and $\mathcal{Z}$ be the set of all finite partitions $-\infty<t_0<t_1<...<t_K\leq \infty$ of the real line. A $U^p$-atom is a piecewise constant function $a:\mathbb{R} \rightarrow \mathcal{H}$ defined by
\begin{align*}
a=\sum_{k=1}^{K}\chi_{[t_{k-1},t_k)}\phi_{k-1},
\end{align*}
where $\{t_k\}_{k=0}^{K} \in \mathcal{Z}$ and $\{\phi_k\}_{k=0}^{K-1} \subset \mathcal{H}$ with $\sum_{k=0}^{K}\|\phi_k\|^p_{\mathcal{H}}=1$. The space $U^p(\mathbb{R};\mathcal{H})$ is then defined as the space of all functions $u:\mathbb{R}\rightarrow \mathcal{H}$ such that $u=\sum_{j=1}^{\infty}\lambda_j a_j$ with $U^p$-atoms $a_j$ and $\{\lambda_j\} \in \ell^1$. We also equip the space $U^p(\mathbb{R};\mathcal{H})$ with the norm
\begin{align*}
\|u\|_{U^p}:=\inf\bg\{\sum^{\infty}_{j=1}|\lambda_j|:u=\sum_{j=1}^{\infty}\lambda_j a_j,\,\lambda_j\in \mathbb{C},\, a_j \text{ are } U^p\textmd{-atoms}\bg\}.
\end{align*}
\end{definition}

\begin{definition}[$V^p$-spaces]
We define the space $V^p(\mathbb{R},\mathcal{H})$ as the space of all functions $v:\mathbb{R} \rightarrow \mathcal{H}$ such that
\begin{align*}
\|v\|_{V^p}:=\sup\limits_{\{t_k\}^K_{k=0} \in \mathcal{Z}}\bg(\sum_{k=1}^{K}\|v(t_k)-v(t_{k-1})\|^p_{\mathcal{H}}\bg)^{\frac{1}{p}} < +\infty,
\end{align*}
where we use the convention $v(\infty)=0$. Also, we denote by $V^p_{rc}(\mathbb{R},\mathcal{H})$ the closed subspace of $V^p(\mathbb{R},\mathcal{H})$ containing all right-continuous functions $v$ with $\lim\limits_{t\rightarrow -\infty}v(t)=0$.
\end{definition}
In our context we shall set the Hilbert space $\mathcal{H}$ to be the Sobolev space $H_{x,y}^s$ with $s\in\R$, which will be the case in the remaining parts of the paper.
\begin{definition}[$U_{\Delta}^p$- and $V_{\Delta}^p$-spaces in \cite{HadacHerrKoch2009}]
For $s\in \mathbb{R}$ let $U^p_{\Delta}H_{x,y}^s(\R)$ resp. $V^p_{\Delta}H_{x,y}^s(\R)$ be the spaces of all functions such that $e^{-it\Delta}u(t)$ is in $U^p(\mathbb{R},H_{x,y}^s)$ resp. $V^p_{rc}(\mathbb{R},H_{x,y}^s)$, with norms
\begin{align*}
\|u\|_{U^p_{\Delta}H_{x,y}^s(\R)}=\|e^{-it\Delta}u\|_{U^p(\mathbb{R},H_{x,y}^s)}, \quad \|u\|_{V^p_{\Delta}H_{x,y}^s(\R)}=\|e^{-it\Delta}u\|_{V^p(\mathbb{R},H_{x,y}^s)}.
\end{align*}
\end{definition}

Having defined the $U_\Delta^p$- and $V_\Delta^p$-spaces we are now ready to formulate the function spaces for studying the Cauchy problems in this paper. For $C=[-\frac{1}{2},\frac{1}{2})^{d+n} \in \mathbb{R}^{d+n}$ and $z\in \mathbb{R}^{d+n}$ let $C_z=z+C$ be the translated unit cube centered at $z$ and define the sharp projection operator $P_{C_z}$ by
\begin{align*}
\mathcal{F}_{x,y}(P_{C_z} f)(\xi,k)=\chi_{C_z}(\xi,k) \mathcal{F}_{x,y} (f)  (\xi,k),\quad(\xi,k)\in\R^d\times\Z^n,
\end{align*}
where $\chi_{C_z}$ is the characteristic function restrained on $C_z$. We then define the $X_0^s$- and $Y^s$-spaces as follows:

\begin{definition}[$X_0^s$- and $Y^s$-spaces]
For $s\in \mathbb{R}$ we define the $X_0^s(\R)$- and $Y^s(\R)$-spaces through the norms
\begin{align*}
\|u\|_{X_0^s(\mathbb{R})}^2&:=\sum_{z\in \mathbb{Z}^{d+n}} \langle z \rangle^{2s} \|P_{C_z} u\|_{U_{\Delta}^2L_{x,y}^2(\R)}^2,\\
\|u\|_{Y^s(\mathbb{R})}^2&:=\sum_{z\in \mathbb{Z}^{d+n}} \langle z \rangle^{2s} \|P_{C_z} u\|_{V_{\Delta}^2L_{x,y}^2(\R)}^2 .
\end{align*}
\end{definition}

For an interval $I \subset \mathbb{R}$ we also consider the restriction spaces $X_0^s(I),Y^s(I)$ etc. For these spaces we have the following useful embedding:

\begin{proposition}[Embedding between the function spaces, \cite{HadacHerrKoch2009}]
For $2< p< q<\infty$ we have
\begin{align*}
U^2_{\Delta}H_{x,y}^s \hookrightarrow X_0^s\hookrightarrow  Y^s \hookrightarrow V^2_{\Delta}H_{x,y}^s\hookrightarrow  U^p_{\Delta}H_{x,y}^s\hookrightarrow U^q_{\Delta}H_{x,y}^s \hookrightarrow L^{\infty}H_{x,y}^s.
\end{align*}
\end{proposition}
Nevertheless, the space $X_0^s$ does not handle the scattering at $t=-\infty$. Thus as in \cite{HaniPausader}, we define the space $X^s(\R)\subset X_0^s(\R)$ by
\[
X^s(\R):=\{u:\,\text{$\phi_{-\infty}:=\lim_{t\to-\infty}e^{-it\Delta}u(t)$ exists in $H_{x,y}^s$ and $u(t)-e^{-it\Delta}\phi_{-\infty}\in X_0^s(\R)$}\}
\]
as our main underlying function space. The space $X^s(I)$ for $I\subset\R$ is similarly defined by its restriction norm. In order to estimate the Duhamel term, we will also need the so-called $N^s$-norm, which is defined as follows:

\begin{definition}[$N^s$-norm]
On a time slot $I$ with $a\in I$ we define the $N^s(I)$-norm for $s\in\R$ by
\begin{equation*}
\| h\|_{N^s(I)}=\|\int_{a}^{t} e^{i(t-s)\Delta} h(s) \, ds \|_{X^s(I)} .
\end{equation*}
\end{definition}

The following proposition reveals the duality of the spaces $N^s(I)$ and $Y^{-s}(I)$.

\begin{proposition}[Duality of $N^s(I)$ and $Y^{-s}(I)$, \cite{HerrTataruTz1,HerrTataruTz2,HaniPausader}]\label{prop:dual}
The spaces $N^s(I)$ and $Y^{-s}(I)$ satisfy the following duality inequality
\begin{align*}
\|f\|_{N^s(I)} \lesssim \sup_{\|v\|_{Y^{-s}(I)}\leq 1} \bg|\int_{I \times  \mathbb{T}^3} f(t,x){v(t,x)} \, dt dx\bg|.
\end{align*}
\end{proposition}

\section{Some useful estimates}\label{bigsec3}
We establish in this section both the deterministic and probabilistic Strichartz estimates and some useful bilinear and nonlinear estimates. For simplicity we also restrict ourselves to the space $\R^3\times\T$.

\subsection{Deterministic and probabilistic Strichartz estimates}

\begin{lemma}[Strichartz estimates on $\R^3\times\T$, \cite{TzvetkovVisciglia2016}]\label{lemma 5.1}
Let $\gamma\in\R$, $\kappa\in[0,3/2)$ and $p,q$ satisfy $p\in[2,\infty]$ and
\begin{align*}
\frac{2}{p}+\frac{3}{q}=\frac{3}{2}-\kappa.
\end{align*}
Then
\begin{align*}
\|e^{it\Delta}f\|_{L_t^p L_x^q H^\gamma_y(\R)}&\lesssim\|f\|_{{H}_x^\kappa H_y^\gamma}.
\end{align*}
\end{lemma}

Following \cite{ShenSofferWu21} we next show the probabilistic Strichartz estimates adapted to the randomization defined in Section \ref{sec random}.

\begin{lemma}[Probabilistic Strichartz estimates]\label{prob strichartz lem}
The following statements hold:
\begin{itemize}
\item[(i)]For any $p\geq 2$ and $s\in\R$ we have
\begin{align}
\|\la\nabla\ra^s e^{it\Delta}f^\omega\|_{L_\omega^p L_t^\infty L_{x,y}^2(\R)}\lesssim\sqrt{p}\|f\|_{H_{x,y}^s}.\label{7.2}
\end{align}

\item[(ii)]
Let $(q,r)$ satisfy $2\leq q,r<\infty$ and $\frac{2}{q}+\frac{3}{r}\leq\frac32$. Let also $(q,r_0)$ be an admissible pair.
%and $K\geq(\frac{1}{2}-\frac{1}{r_0})/ (\frac{3}{r_0}-\frac{3}{r})$.
Then for $s\leq
K(\frac{3}{r_0}-\frac{3}{r})-(\frac{1}{2}-\frac{1}{r_0})$ and $p\geq\max\{q,r\}$ we have
\begin{align}
\|\la\nabla\ra^s e^{it\Delta}f^\omega\|_{L_\omega^p L_t^q L_{x,y}^r(\R)}\lesssim\sqrt{p}\|f\|_{L_{x,y}^2}.\label{7.3}
\end{align}

\item[(iii)]
Let $q$ satisfy $2\leq q< \infty$. Let also $(q,r_0)$ be an admissible pair.
% and $K\geq\frac{r_0}{6}-\frac{1}{3}$.
Then for $s<
\frac{3K}{r_0}-(\frac{1}{2}-\frac{1}{r_0})$ there exists some $p_0\geq 2$ such that for all $p\geq p_0$ we have
\begin{align}
\|\la\nabla\ra^s e^{it\Delta}f^\omega\|_{L_\omega^p L_t^q L_{x,y}^\infty(\R)}\lesssim\sqrt{p}\|f\|_{L_{x,y}^2}.\label{7.3+}
\end{align}

\item[(iv)]
Let $r$ satisfy $2<r\leq \infty$. Then for $s<
K(\frac{3}{2}-\frac{3}{r})$ there exists some $p_0\geq 2$ such that for all $p\geq p_0$ we have
\begin{align}
\|\la\nabla\ra^s e^{it\Delta}f^\omega\|_{L_\omega^p L_t^\infty L_{x,y}^r(\R)}\lesssim\sqrt{p}\|f\|_{L_{x,y}^2}.\label{7.3++}
\end{align}
\end{itemize}
\end{lemma}

\begin{proof}
We begin with the proof of \eqref{7.2}. Using Plancherel, Minkowski and the unitarity of $e^{it\Delta}$ we obtain
\[
\|\la\nabla\ra^se^{it\Delta}f^\omega\|_{L_\omega^p L_t^\infty L_{x,y}^2(\R)}\lesssim \|\la\nabla\ra^sf^\omega\|_{L_\omega^p L_{x,y}^2}
\leq \|\la\nabla\ra^sf^\omega\|_{L_{x,y}^2 L_\omega^p }.
\]

Combining with Lemma \ref{lemma 6.1} and Lemma \ref{lemma 3.1} we infer that
\begin{align*}
\|\la\nabla\ra^sf^\omega\|_{L_{x,y}^2 L_\omega^p }\lesssim \sqrt{p}\|\la\nabla\ra^s(\Box_{jk}f_k e^{iky})\|_{L_{x,y}^2\ell_{jk}^2}
\sim\sqrt p\|f\|_{H_{x,y}^s},
\end{align*}
which gives \eqref{7.2}. Next we prove \eqref{7.3}. Using Minkowski, Lemma \ref{lemma 6.1}, Lemma \ref{lemma 3.2}, the embedding $L_y^{r_0}\hookrightarrow L_y^2$, Lemma \ref{lemma 5.1} and Lemma \ref{lemma 3.1} we obtain that
\begin{align}
\|\la\nabla\ra^se^{it\Delta}f^\omega\|_{L_\omega^p L_t^q L_{x,y}^r(\R)}
&\lesssim
\|\la\nabla\ra^se^{it\Delta}f^\omega\|_{ L_t^q L_{x,y}^r L_\omega^p(\R)}
\lesssim
\sqrt{p}\|\la\nabla\ra^se^{it\Delta}\Box_{jk}f_ke^{iky}\|_{L_t^q L_{x,y}^r \ell_{jk}^2(\R)}\nonumber\\
&\lesssim
\sqrt{p}\|\la\nabla\ra^se^{it\Delta}\Box_{jk}f_ke^{iky}\|_{\ell_{jk}^2 L_t^q L_{x,y}^r (\R)}\nonumber\\
&\lesssim
\sqrt{p}\|\la\nabla\ra^{s-K(\frac{3}{r_0}-\frac{3}{r})}e^{it\Delta}\Box_{jk}f_ke^{iky}\|_{\ell_{jk}^2 L_t^q L_{x,y}^{r_0} (\R)}\nonumber\\
&\lesssim
\sqrt{p}\|\la\nabla\ra^{s-K(\frac{3}{r_0}-\frac{3}{r})+(\frac{1}{2}-\frac{1}{r_0})}e^{it\Delta}\Box_{jk}f_ke^{iky}\|_{\ell_{jk}^2 L_t^q L_{x}^{r_0}L_y^2 (\R)}\nonumber\\
&\lesssim
\sqrt{p}\|\la\nabla\ra^{s-K(\frac{3}{r_0}-\frac{3}{r})+(\frac{1}{2}-\frac{1}{r_0})}\Box_{jk}f_ke^{iky}\|_{\ell_{jk}^2 L_{x,y}^2(\R)}\nonumber\\
&\lesssim\sqrt{p}\|\Box_{jk}f_ke^{iky}\|_{\ell_{jk}^2 L_{x,y}^2(\R)}
\sim\sqrt{p}\|f\|_{L_{x,y}^2}.\label{7.5}
\end{align}
Now we prove \eqref{7.3+}. Let $\vare\in(0,\frac{1}{K+2}(\frac{3K}{r_0}-(\frac12-\frac{1}{r_0})-s)]$ so that
\[s+2\vare \leq K\bg(\frac{3}{r_0}-\frac{3}{4/\vare}\bg)-\bg(\frac{1}{2}-\frac{1}{4/\vare}\bg).\]
Using the Sobolev embedding $W_{x,y}^{2\vare,\frac{4}{\vare}}\hookrightarrow L_{x,y}^\infty$ we infer that
\begin{align*}
\|\la\nabla\ra^se^{it\Delta}f^\omega\|_{L_\omega^p L_t^q L_{x,y}^\infty(\R)}
\lesssim
\|\la\nabla\ra^{s+2\vare}e^{it\Delta}f^\omega\|_{L_\omega^p L_t^q L_{x,y}^{\frac{4}{\vare}}(\R)}.
\end{align*}
By setting $p_0:=\max\{q,\frac{4}{\vare}\}$, the proof of \eqref{7.3+} follows from using the similar arguments as in \eqref{7.5}, we omit the repeating details here. Finally we prove \eqref{7.3++}. It suffices to consider the case $r=\infty$, the general case follows then from interpolating with \eqref{7.2}. Using the Sobolev embedding $W_{t}^{2\vare,\frac{1}{\vare}}\hookrightarrow L_t^\infty$ and $W_{x,y}^{2\vare,\frac{4}{\vare}}\hookrightarrow L_{x,y}^\infty$ for $\vare>0$ we obtain
\begin{align*}
\|\la\nabla\ra^se^{it\Delta}f^\omega\|_{L_\omega^p L_t^\infty L_{x,y}^\infty(\R)}
\lesssim
\|\la\nabla\ra^{s+2\vare}\la\pt_t\ra^{2\vare}e^{it\Delta}f^\omega\|_{L_\omega^p L_t^{\frac{1}{\vare}} L_{x,y}^{\frac{4}{\vare}}(\R)}
\lesssim
\|\la\nabla\ra^{s+6\vare}e^{it\Delta}f^\omega\|_{L_\omega^p L_t^{\frac{1}{\vare}} L_{x,y}^{\frac{4}{\vare}}(\R)}.
\end{align*}
The remaining proof can be deduced similarly as the one for \eqref{7.3+}, we thus omit the details here.
\end{proof}

The following corollary is an immediate consequence of Lemma \ref{2.4 lem} and Lemma \ref{prob strichartz lem}
\begin{lemma}[Almost sure finiteness of crucial norms]\label{finiteness lemma}
Let $s\in\R$, $f\in H_{x,y}^s$ and suppose that $K\in\N$ satisfies $K>-6s+\frac23$. Then for $f^\omega$ defined by \eqref{1.7} associated with $f$ we have
\[\mathbb{P}(\|f^\omega\|_{H_{x,y}^s}+\|f^\omega\|_{L_{x,y}^4}+\|e^{it\Delta}f^\omega\|_{W(\R)}<\infty)=1.\]
\begin{proof}
Notice that $K\in\N$. Thus $K>-6s+\frac23$ actually implies
\[K>\max\{-2s+\frac23,-4s+\frac23,-6s+\frac23,-\frac{4s}{3}\}.\]
Hence Lemma \ref{prob strichartz lem} is applicable and we infer that there exists some $p_0\geq 2$ such that for all $p\geq p_0$
\begin{align*}
&\,\|f^\omega\|_{L_\omega^p H_{x,y}^s}+\|f^\omega\|_{L_\omega^p L_{x,y}^4(\R)}+\|\la\nabla\ra^{(s-\frac13+\frac{K}{2})-}e^{it\Delta}f^\omega\|_{L_\omega^pL_t^2 L_{x,y}^\infty(\R)}\nonumber\\
&\,+\|e^{it\Delta}f^\omega\|_{L_\omega^p L_{t,x,y}^4(\R)}+\|e^{it\Delta}f^\omega\|_{L_\omega^p L_t^6 L_{x,y}^3(\R)}\nonumber\\
\lesssim &\,\|f^\omega\|_{L_\omega^p H_{x,y}^s}+\|\la\nabla\ra^{(s+\frac{3 K}{4})-}e^{it\Delta}f^\omega\|_{L_\omega^p L_t^\infty L_{x,y}^4(\R)}+\|\la\nabla\ra^{(s-\frac13+\frac{K}{2})-}e^{it\Delta}f^\omega\|_{L_\omega^pL_t^2 L_{x,y}^\infty(\R)}\nonumber\\
&\,+\|\la\nabla\ra^{s-\frac16+\frac{K}{4}}e^{it\Delta}f^\omega\|_{L_\omega^pL_{t,x,y}^4(\R)}
+\|\la\nabla\ra^{s-\frac19+\frac{K}{6}}e^{it\Delta}f^\omega\|_{L_\omega^pL_t^6 L_{x,y}^3(\R)}\lesssim\sqrt{p}\|f\|_{H_{x,y}^s}.
\end{align*}
The desired claim then follows from Lemma \ref{2.4 lem}.
\end{proof}

\end{lemma}

\subsection{Bilinear and nonlinear estimates}\label{sec 3.2}
We define the $Z$-norm by
\[
\|u\|_{Z(I)}:=\bg(\sum_{N\geq 1}N^{6-p_0}\|\chi_{t\in I}P_N u\|^{p_0}_{\ell_k^{\frac{2p_0}{p_0-3}}L_{t,x,y}^{p_0}(k,k+1)}\bg)^{\frac{1}{p_0}}
\]
with $p_0=\frac{21}{4}$. The $Z$-norm is referred to as the scattering norm for the cubic NLS on $\R^3\times\T$. Moreover, it is weaker than the $X^1$-norm, i.e. $\|u\|_Z\lesssim \|u\|_{X^1}$. For more details of the properties of the $Z$-norm, we refer to \cite{HaniPausader,RmT1}. We shall also use the $Z$-norm to build up our local theory. To that end we need the following useful lemma concerning the estimate of the Duhamel integral given in term of the $Z$-norm.

\begin{lemma}[Nonlinear estimate I, \cite{RmT1}]\label{9.1 lem}
Define the $Z'$-norm by
\[\|u\|_{Z'}:=\|u\|_{Z}^{\frac34}\|u\|_{X^1}^{\frac14}.\]
Then
\begin{align*}
\|w_1 w_2 w_3\|_{N^1}\lesssim \sum_{\{i,j,k\}=\{1,2,3\}}\|w_i\|_{X^1}\|w_j\|_{Z'}\|w_k\|_{Z'}.
\end{align*}
\end{lemma}

We will also need a suitable estimate for the term $\|v \tilde w^2 \|_{N^1}$, where $v$ and $\tilde w$ are the functions given in \eqref{4.49} below. In particular, the upper bound of this term should not contain any term of the form $\|v_0\|_{H_{x,y}^\alpha}$ with $\alpha>s$. To derive such an upper bound, we firstly state some auxiliary bilinear estimates.

\begin{lemma}[Bilinear estimate I, \cite{HerrTataruTz2}]\label{bil lem 1}
There exists some $\delta>0$ such that for all $N_1\geq N_2$ we have
\begin{align*}
\|P_{N_1}u_1 P_{N_2}u_2\|_{L_{t,x,y}^2}\lesssim N_2\bg(\frac{N_2}{N_1}+\frac{1}{N_2}\bg)^\delta
\|P_{N_1}u_1 \|_{Y^0}\|P_{N_2}u_2\|_{Y^0}.
\end{align*}
\end{lemma}

\begin{lemma}[Bilinear estimate II, \cite{RmT1}]\label{bil lem 2}
There exists some $\delta>0$ such that for all $N_1\geq N_2$ we have
\begin{align*}
\|P_{N_1}u_1 P_{N_2}u_2\|_{L_{t,x,y}^2}\lesssim \bg(\frac{N_2}{N_1}+\frac{1}{N_2}\bg)^\delta
\|P_{N_1}u_1 \|_{Y^0}\|P_{N_2}u_2\|_{Z'}.
\end{align*}
\end{lemma}

\begin{lemma}[Bilinear estimate III]\label{bil lem 3}
Let $s\in\R$, $\kappa>-2s+1$ and $v_0\in H^s_{x,y}$. Set also $\beta:=2s+\kappa-1>0$ and $v=e^{it\Delta}v_0$. Then there exists some $\delta>0$ such that
\begin{itemize}
\item[(i)]For $N_1\geq N_2\geq 1$ we have
\begin{align*}
\|P_{N_1} v P_{N_2}u\|_{L_{t,x,y}^2}&\lesssim \bg(\frac{N_2}{N_1^{2s+\kappa}}\bg)^{\frac12}
\bg(\frac{N_2}{N_1}+\frac{1}{N_2}\bg)^{\frac{\delta}{2}}
\|P_{N_1}v_0\|^{\frac12}_{H_{x,y}^s}\|\la\nabla\ra^{s+\kappa}P_{N_1}v\|^{\frac12}_{L_t^2 L_{x,y}^\infty}\|P_{N_2}u\|_{Y^0}\nonumber\\
&\lesssim N_1^{-\frac{\beta}{2}}
\bg(\frac{N_2}{N_1}+\frac{1}{N_2}\bg)^{\frac{\delta}{2}}
\|P_{N_1}v_0\|^{\frac12}_{H_{x,y}^s}\|\la\nabla\ra^{s+\kappa}P_{N_1}v\|^{\frac12}_{L_t^2 L_{x,y}^\infty}\|P_{N_2}u\|_{Y^0}.
\end{align*}

\item[(ii)]For $1\leq N_1\leq N_2$ we have
\begin{align*}
\|P_{N_1} v P_{N_2}u\|_{L_{t,x,y}^2}&\lesssim \bg(\frac{1}{N_1^{2s+\kappa-1}}\bg)^{\frac12}
\bg(\frac{N_1}{N_2}+\frac{1}{N_1}\bg)^{\frac{\delta}{2}}
\|P_{N_1}v_0\|^{\frac12}_{H_{x,y}^s}\|\la\nabla\ra^{s+\kappa}P_{N_1}v\|^{\frac12}_{L_t^2 L_{x,y}^\infty}\|P_{N_2}u\|_{Y^0}\nonumber\\
&\lesssim
\bg(\frac{N_1}{N_2}+\frac{1}{N_1}\bg)^{\frac{\delta}{2}}
\|P_{N_1}v_0\|^{\frac12}_{H_{x,y}^s}\|\la\nabla\ra^{s+\kappa}P_{N_1}v\|^{\frac12}_{L_t^2 L_{x,y}^\infty}\|P_{N_2}u\|_{Y^0}.
\end{align*}
\end{itemize}
\end{lemma}

\begin{proof}
We only prove (i), the proof of (ii) follows in a similar way. Notice that by definition and the support property of $P_{N_1}$ we have $\|P_{N_1} v\|_{Y^0}\lesssim N_1^{-s}\|P_{N_1}v_0\|_{H_{x,y}^s}$. Thus in view of Lemma \ref{bil lem 1} and interpolation, it suffices to show
\[
\|P_{N_1} v P_{N_2}u\|_{L_{t,x,y}^2}\lesssim N_1^{-s-\kappa}\|\la\nabla\ra^{s+\kappa}P_{N_1}v\|_{L_t^2 L_{x,y}^\infty}\|P_{N_2}u\|_{Y^0}.
\]
This however follows immediately from the H\"older's and Bernstein's inequalities and the embedding $Y^0 \hookrightarrow L_t^\infty L_{x,y}^2$
\begin{align*}
\|P_{N_1} v P_{N_2}u\|_{L_{t,x,y}^2}&\lesssim \|P_{N_1} v \|_{L_t^2 L_{x,y}^\infty}\|P_{N_2}u\|_{L_t^\infty L_{x,y}^2}\nonumber\\
&\lesssim N_1^{-s-\kappa}\|\la\nabla\ra^{s+\kappa}P_{N_1}v\|_{L_t^2 L_{x,y}^\infty}\|P_{N_2}u\|_{Y^0}.
\end{align*}
The proof is therefore complete.
\end{proof}

We now derive a suitable upper bound for the nonlinear term $\|v w^2\|_{N^1}$ with $v=e^{it\Delta}v_0$ and $v_0\in H_{x,y}^s$.
\begin{lemma}[Nonlinear estimate II]\label{5.5 lem}
Let $s\in\R$, $\kappa>-2s+3$ and $v_0\in H^s_{x,y}$. Set also that $\beta:=2s+\kappa-1>2$ and $v=e^{it\Delta}v_0$. Then
\begin{align*}
\|v w^2\|_{N^1}\lesssim\|v_0\|^{\frac12}_{H_{x,y}^s}\|\la\nabla\ra^{s+\kappa} v\|^{\frac12}_{L_t^2 L_{x,y}^\infty}\|w\|_{X^1}^2.
\end{align*}
\end{lemma}

\begin{proof}
Let $u$ satisfy $\|u\|_{Y^{-1}}\leq 1$. Write
\[vw^2u=\sum_{N_0,\cdots, N_3\geq 1}P_{N_0}u P_{N_1}v P_{N_2}w P_{N_3}w=:
\sum_{N_0,\cdots, N_3\geq 1}P_0 u P_1 v P_2 w P_3 w.\]
By symmetry we may assume $N_2\leq N_3$. By the support properties of the Littlewood-Paley decomposition it suffices to consider the following cases:
\begin{align*}
\text{(i) }&\text{$N_0\leq N_3$, $N_1\leq N_2\leq N_3$, $N_3\sim \max\{N_0,N_2\}$},\\
\text{(ii) }&\text{$N_0\leq N_3$, $N_2\leq N_1\leq N_3$, $N_3\sim \max\{N_0,N_1\}$},\\
\text{(iii) }&\text{$N_0\leq N_1$, $N_2\leq N_3\leq N_1$, $N_1\sim \max\{N_0,N_3\}$}.
\end{align*}
We prove the estimates case by case using the bilinear estimates.
\begin{itemize}
\item[(ia)]$N_0\sim N_3\geq N_2\geq N_1$. Using the bilinear estimates and Cauchy-Schwarz we obtain
\begin{align*}
&\lesssim \sum_{N_0\sim N_3\geq N_2\geq N_1}
\bg(\frac{N_1}{N_0}+\frac{1}{N_1}\bg)^{\frac{\delta}{2}}
\bg(\frac{N_2}{N_3}+\frac{1}{N_2}\bg)^{\frac{\delta}{2}}
\nonumber\\
&\times\|P_{1}v_0\|^{\frac12}_{H_{x,y}^s}\|\la\nabla\ra^{s+\kappa}P_{1}v\|^{\frac12}_{L_t^2 L_{x,y}^\infty}\|P_{0}u\|_{Y^0}\|P_{2}w\|_{Z'}\|P_{3}w\|_{Y^0}\nonumber\\
&\lesssim \|v_0\|^{\frac12}_{H_{x,y}^s}\|\la\nabla\ra^{s+\kappa}v\|^{\frac12}_{L_t^2 L_{x,y}^\infty}\|w\|_{Z'}\sum_{N_0\sim N_3}\|P_0 u\|_{Y^{-1}}
\|P_3 w\|_{Y^1}\nonumber\\
&\lesssim \|v_0\|^{\frac12}_{H_{x,y}^s}\|\la\nabla\ra^{s+\kappa}v\|^{\frac12}_{L_t^2 L_{x,y}^\infty}\|w\|_{Z'}\|w\|_{X^1}\|u\|_{Y^{-1}}\nonumber\\
&\lesssim \|v_0\|^{\frac12}_{H_{x,y}^s}\|\la\nabla\ra^{s+\kappa}v\|^{\frac12}_{L_t^2 L_{x,y}^\infty}\|w\|^2_{X^1}.
\end{align*}

\item[(ib)]$N_3\sim N_2\geq N_0, N_1$. We have
\begin{align*}
&\lesssim \sum_{N_3\sim N_2\geq N_0, N_1}
N_0\bg(\frac{N_0}{N_2}+\frac{1}{N_0}\bg)^{\frac{\delta}{2}}
\bg(\frac{N_1}{N_3}+\frac{1}{N_1}\bg)^{\frac{\delta}{2}}
\nonumber\\
&\times\|P_{1}v_0\|^{\frac12}_{H_{x,y}^s}\|\la\nabla\ra^{s+\kappa}P_{1}v\|^{\frac12}_{L_t^2 L_{x,y}^\infty}
\|P_{3}w\|_{Y^0}\|P_{0}u\|_{Y^0}\|P_{2}w\|_{Y^0}\nonumber\\
&\lesssim \|v_0\|^{\frac12}_{H_{x,y}^s}\|\la\nabla\ra^{s+\kappa}v\|^{\frac12}_{L_t^2 L_{x,y}^\infty}\|u\|_{Y^{-1}}\sum_{N_3\sim N_2}\|P_2 w\|_{Y^{1}}
\|P_3 w\|_{Y^1}\nonumber\\
&\lesssim \|v_0\|^{\frac12}_{H_{x,y}^s}\|\la\nabla\ra^{s+\kappa}v\|^{\frac12}_{L_t^2 L_{x,y}^\infty}\|w\|^2_{X^1}.
\end{align*}

\item[(iia)]$N_0\sim N_3\geq N_1\geq N_2$. We have
\begin{align*}
&\lesssim \sum_{N_0\sim N_3\geq N_1\geq N_2}
\bg(\frac{N_2}{N_0}+\frac{1}{N_2}\bg)^{\frac{\delta}{2}}
\bg(\frac{N_1}{N_3}+\frac{1}{N_1}\bg)^{\frac{\delta}{2}}
\nonumber\\
&\times\|P_{1}v_0\|^{\frac12}_{H_{x,y}^s}\|\la\nabla\ra^{s+\kappa}P_{1}v\|^{\frac12}_{L_t^2 L_{x,y}^\infty}\|P_{3}w\|_{Y^0}\|P_{0}u\|_{Y^0}\|P_{2}w\|_{Z'}\nonumber\\
&\lesssim \|v_0\|^{\frac12}_{H_{x,y}^s}\|\la\nabla\ra^{s+\kappa}v\|^{\frac12}_{L_t^2 L_{x,y}^\infty}\|w\|^2_{X^1}.
\end{align*}

\item[(iib)]$N_1\sim N_3\geq N_2, N_0$. We have
\begin{align*}
&\lesssim \sum_{N_1\sim N_3\geq N_2, N_0}
N_1^{-\frac{\beta}{2}}\bg(\frac{N_0}{N_1}+\frac{1}{N_0}\bg)^{\frac{\delta}{2}}
\bg(\frac{N_2}{N_3}+\frac{1}{N_2}\bg)^{\frac{\delta}{2}}
\nonumber\\
&\times\|P_{1}v_0\|^{\frac12}_{H_{x,y}^s}\|\la\nabla\ra^{s+\kappa}P_{1}v\|^{\frac12}_{L_t^2 L_{x,y}^\infty}\|P_{0}u\|_{Y^0}\|P_{2}w\|_{Z'}\|P_{3}w\|_{Y^0}\nonumber\\
&\lesssim \|v_0\|^{\frac12}_{H_{x,y}^s}\|\la\nabla\ra^{s+\kappa}v\|^{\frac12}_{L_t^2 L_{x,y}^\infty}\|u\|_{Y^{-1}}
\sum_{N_1\sim N_3}N_1^{-\frac{\beta}{2}} \|P_3 w\|_{Y^1}\nonumber\\
&\lesssim \|v_0\|^{\frac12}_{H_{x,y}^s}\|\la\nabla\ra^{s+\kappa}v\|^{\frac12}_{L_t^2 L_{x,y}^\infty}\|w\|^2_{X^1}.
\end{align*}

\item[(iiia)]$N_1\sim N_0\geq N_3\geq N_2$. We have
\begin{align*}
&\lesssim \sum_{N_1\sim N_0\geq N_3\geq N_2}
N_1^{-\frac{\beta}{2}}\bg(\frac{N_2}{N_1}+\frac{1}{N_2}\bg)^{\frac{\delta}{2}}
\bg(\frac{N_3}{N_0}+\frac{1}{N_3}\bg)^{\frac{\delta}{2}}
\nonumber\\
&\times\|P_{1}v_0\|^{\frac12}_{H_{x,y}^s}\|\la\nabla\ra^{s+\kappa}P_{1}v\|^{\frac12}_{L_t^2 L_{x,y}^\infty}\|P_{2}w\|_{Y^0}\|P_{3}w\|_{Z'}\|P_{0}w\|_{Y^0}\nonumber\\
&\lesssim \|v_0\|^{\frac12}_{H_{x,y}^s}\|\la\nabla\ra^{s+\kappa}v\|^{\frac12}_{L_t^2 L_{x,y}^\infty}\|w\|_{X^1}^2\sum_{N_1\sim N_0}N_1^{1-\frac{\beta}{2}}\|P_{0}w\|_{Y^{-1}}\nonumber\\
&\lesssim \|v_0\|^{\frac12}_{H_{x,y}^s}\|\la\nabla\ra^{s+\kappa}v\|^{\frac12}_{L_t^2 L_{x,y}^\infty}\|w\|^2_{X^1},
\end{align*}
where we also used the fact that $\beta>2$.

\item[(iiib)]$N_1\sim N_3\geq N_2,N_0$. We have
\begin{align*}
&\lesssim \sum_{N_1\sim N_3\geq N_2,N_0}
N_1^{-\frac{\beta}{2}}\bg(\frac{N_2}{N_3}+\frac{1}{N_2}\bg)^{\frac{\delta}{2}}
\bg(\frac{N_0}{N_1}+\frac{1}{N_0}\bg)^{\frac{\delta}{2}}
\nonumber\\
&\times\|P_{1}v_0\|^{\frac12}_{H_{x,y}^s}\|\la\nabla\ra^{s+\kappa}P_{1}v\|^{\frac12}_{L_t^2 L_{x,y}^\infty}\|P_{0}u\|_{Y^0}\|P_{2}w\|_{Z'}\|P_{3}w\|_{Y^0}\nonumber\\
&\lesssim \|v_0\|^{\frac12}_{H_{x,y}^s}\|\la\nabla\ra^{s+\kappa}v\|^{\frac12}_{L_t^2 L_{x,y}^\infty}\|w\|_{Z'}\|u\|_{Y^{-1}}\sum_{N_1\sim N_3}N_1^{-\frac{\beta}{2}}\|P_{3}w\|_{Y^{1}}\nonumber\\
&\lesssim \|v_0\|^{\frac12}_{H_{x,y}^s}\|\la\nabla\ra^{s+\kappa}v\|^{\frac12}_{L_t^2 L_{x,y}^\infty}\|w\|^2_{X^1}.
\end{align*}
\end{itemize}
The desired claim then follows from Proposition \ref{prop:dual}.
\end{proof}

\section{Proof of Theorem \ref{main thm}}\label{bigsec4}
In this section we give the proof of Theorem \ref{main thm}. We collect some important auxiliary results given in Section \ref{sec 4.1} to Section \ref{sec 4.4}. Having all the preliminaries we are then able to give the complete proof of Theorem \ref{main thm} in Section \ref{sec 4.5}.
%By the embedding $H_{x,y}^{s_1}\hookrightarrow H_{x,y}^{s_2}$ for $s_1>s_2$ we will also only consider the (harder) case $s<0$, where even the mass conservation of the NLS is \textit{a priori} unavailable.

\subsection{Local theory}\label{sec 4.1}
Assume that a solution $u$ of \eqref{nls} with initial data $u_0$ is decomposed to $u=v+w$, $u_0=v_0+w_0$ with $v=e^{it\Delta}v_0$. Then $w$ solves the NLS
\begin{align}\label{nls_w}
(i\pt_t+\Delta_{x,y})w=|u|^2 u=|w|^2 w+e
\end{align}
with $w(0)=w_0$ and $e:=|u|^2u-|w|^2w$. In the following we derive a suitable local theory for the problem \eqref{nls_w}, which will be serving as the starting point of our inductive proof of Theorem \ref{main thm}.

\begin{lemma}[Local theory]\label{local theo lem}
Let $w_0\in H_{x,y}^1$ and $v_0\in H_{x,y}^s$ with $s\in\R$. Suppose that there exists some $\kappa>\max\{-2s+3,1-s\}$ such that \[\|v\|_{L_{t,x,y}^4}+\|\la\nabla\ra^{s+\kappa} v\|_{L_t^2 L_{x,y}^\infty}<\infty,\]
where $v=e^{it\Delta}v_0$. Then there exists some $T>0$ depending on $\kappa,w_0,v_0$ such that \eqref{nls_w} possesses a solution
$w\in X^1(-T,T)$.
\end{lemma}

\begin{remark}\label{remark 4.2}
In our case the condition $\kappa>1-s$ is redundant. Indeed, we will set $\kappa=\frac{K}{2}-\frac13$ in view of the $W$-norm defined by \eqref{w norm}. This requires particularly that $K>\max\{-4s+\frac{20}{3},-2s+\frac83\}$. However, for $K\in\N$ and $K>-4s+\frac{20}{3}$ stated as assumptions in Theorem \ref{main thm}, the latter condition is always fulfilled.
\end{remark}

\begin{proof}
In the following all space-time norms are taken over the interval $I=(-T,T)$ and thus the dependence of such norms on $I$ will be neglected in the upcoming calculations. Let $C_0>0$ be given such that $\|e^{it\Delta}w_0\|_{X^1}\leq C_0$ (the $H_{x,y}^1$-norm of $w_0$ is also merged to the constant $C_0$). By shrinking $T$ to zero we may also assume that $\|e^{it\Delta}w_0\|_{Z'}\leq \delta$ for some to be determined $\delta$. Define the space $S(I)$ by
\begin{align*}
S(I):=\{u\in X^1:\|u\|_{X^1}\leq 2C_0,\,\|u\|_{Z'}\leq 2\delta\}.
\end{align*}
Define also the contraction mapping $\Phi$ by
\[\Phi(w)(t):=e^{it\Delta}w_0-i\int_0^t e^{i(t-s)\Delta}(|w+v|^2 (w+v))(s)\,ds.\]
We aim to show that $\Phi$ defines a contraction on $S(I)$, from which the desired claim follows by combining with the Banach fixed point theorem. Using Lemma \ref{prop:dual} and the embedding $X^1\hookrightarrow Z'$ applied to the Duhamel term we first obtain
\begin{align*}
\|\Phi(w)\|_{X^1}&\leq \|e^{it\Delta}w_0\|_{X^1}+\||w+v|^2(w+v)\|_{N^1}
\leq C_0+C\|v^3\|_{N^1}+C\|w^3\|_{N^1},\\
\nonumber\\
\|\Phi(w)\|_{Z'}&\leq \|e^{it\Delta}w_0\|_{Z'}+C\||w+v|^2(w+v)\|_{N^1}
\leq \delta+C\|v^3\|_{N^1}+C\|w^3\|_{N^1}.
\end{align*}
Using the H\"older's inequality
\begin{align*}\|\la\nabla\ra v_1(v_2v_3v_4)\|_{L_{t,x,y}^1}\leq
\|\la\nabla\ra v_1\|_{L_t^2 L_{x,y}^\infty}\|v_2\|_{L_{t,x,y}^4}\|v_3\|_{L_{t,x,y}^4}\|v_4\|_{L_t^\infty L_{x,y}^2}
\end{align*}
and the embedding $X^\alpha\hookrightarrow L_t^\infty H_{x,y}^\alpha$ we obtain
\begin{align*}
\|v^3\|_{N^1}&\lesssim \|\la\nabla\ra v\|_{L_t^2L_{x,y}^\infty}\|v\|^2_{L_{t,x,y}^4}.
\end{align*}
Using Lemma \ref{9.1 lem} we also infer that
\begin{align*}
\|w^3\|_{N^1}&\lesssim \|w\|_{Z'}^2\|w\|_{X^1}.
\end{align*}
We now shrink $T$ if necessary such that
\[\|\la\nabla\ra  v\|_{L_t^2 L_{x,y}^\infty}\lesssim \|\la\nabla\ra^{s+\kappa} v\|_{L_t^2 L_{x,y}^\infty}<\delta^2,\]
where we also used the fact that $s+\kappa>1$. Then
\begin{align*}
\|\Phi(w)\|_{X^1}\leq C_0+C\delta^2(1+C_0)^2,\quad
\|\Phi(w)\|_{Z'}\leq \delta+C\delta^2(1+C_0)^2.
\end{align*}
Thus it follows that $\Phi$ maps $S(I)$ to $S(I)$ when choosing $\delta$ and then $T$ sufficiently small. The proof for showing that $\Phi$ is a contraction is similar, we omit the details here.
\end{proof}

\subsection{Interaction Morawrtz inequality}\label{sec 4.2}
Our goal in this subsection is to derive an interaction Morawetz inequality for the solution $w$ of \eqref{nls_w}, which is stated in next lemma. Since the interaction Morawetz inequality is involved with four variables, following the convention in literature we shall refer $x,y$ to as the $\R^3$-variables, while $z,\zt$ are the $\T$-ones.
\begin{lemma}[Interaction Morawetz inequality]\label{inter mora lem}
Let $w$ be a solution of \eqref{nls_w}. Then for a given time slot $J\subset\R$ we have
\begin{align*}
\|w\|_{L_{t,x}^4 L_z^2}^4&\lesssim \|w\|^2_{L_t^\infty L_{x,z}^2}\|w\|_{L_t^\infty H_x^{\frac12}L_z^2}^2\nonumber\\
&+\|v\|_{L_t^2 L_{x,z}^\infty}\|v\|^2_{L_{t,x,z}^4}
\| w\|_{L_t^\infty L_x^2 H_z^{\frac12+}}\left(\| w\|_{L_t^\infty L_{x,y}^2}\| w\|_{L_t^\infty H_{x,z}^1}
+\| w\|^2_{L_t^\infty L_x^2 H_z^{\frac12+}}\right)
\nonumber\\
&+\|v\|_{L_t^2 L_{x,z}^\infty}\|w\|^2_{L_{t,x}^4L_z^2}
\| w\|_{L_t^\infty L_x^2 H_z^{\frac12+}}\left(\| w\|_{L_t^\infty L_{x,y}^2}\| w\|_{L_t^\infty H_{x,z}^1}
+\| w\|^2_{L_t^\infty L_x^2 H_z^{\frac12+}}\right)
\nonumber\\
&+\|v\|_{L_t^2 L_{x,y}^\infty} \|v\|^3_{L_t^6 L_{x,y}^3}
\|w\|^2_{L_t^\infty H_x^{\frac12}L_z^2}
+\|v\|_{L_t^2 L_{x,y}^\infty}\|w\|_{L_{t,x}^4 L_z^2}^2\| w\|^3_{L_t^\infty L_x^2 H_z^{\frac12+}}\nonumber\\
&+\|\la\nabla\ra v\|_{L_t^2 L_{x,y}^\infty} \|v\|^3_{L_t^6 L_{x,y}^3}
\|w\|_{L_{t}^\infty L_{x,z}^2}^2
+\|\la\nabla\ra v\|_{L_t^2 L_{x,z}^\infty}\| w\|^2_{L_{t,x}^4L_z^2}\|w\|_{L_{t}^\infty L_{x,z}^2}^2\|w\|_{L_t^\infty L_x^2 H_z^{\frac12+}},
\end{align*}
where all space-time norms are taken over the time slot $J$.
\end{lemma}

\begin{proof}
Define the interaction Morawetz potential $I(t)$ by
\begin{align*}
I(t):=2\int_{(\R^3\times\T)^2}|w(t,x,z)|^2\frac{x-y}{|x-y|}\cdot\im(\bar w(t,y,\zt)\nabla_x w(t,y,\zt))\,dxdydzd\zt.
\end{align*}
In the case where $w$ is independent of $z$, we have the well-known Morawetz identity (see \cite{intermora})
\begin{align}
\pt_t I(t)&=8\pi\|w(t)\|^4_{L_x^4L_z^2}\label{8.3}+2\int\frac{1}{|x-y|}|w(t,x,z)|^4|w(t,y,\zt)|^2\,dxdydzd\zt\\
&+4\int\bg(\frac{\delta_{ij}}{|x|}-\frac{x_ix_j}{|x|^3}\bg)\re(\pt_i\bar w \pt_j w)(t,x,z)|w(t,y,\zt)|^2 \,dxdydzd\zt
\label{8.5}\\
&-4\int\bg(\frac{\delta_{ij}}{|x|}-\frac{x_ix_j}{|x|^3}\bg)\im(\bar w\pt_j w)(t,x,z)\im(\bar w\pt_k w)(t,y,\zt)\,dxdydzd\zt
\label{8.6}\\
&+4\int\frac{x-y}{|x-y|}\cdot \im(\bar w\nabla w)(t,x,z)\im(\bar w e)(t,y,\zt)\,dxdydzd\zt\label{8.7}\\
&+4\int\frac{x-y}{|x-y|}\cdot\re(\bar e\nabla w)(t,x,z)|w(t,y,\zt)|^2\,dxdydzd\zt
\label{8.7+}\\
&+4\int\frac{1}{|x-y|}\re(\bar ew )(t,x,z)|w(t,y,\zt)|^2\,dxdydzd\zt
\label{8.8}
\end{align}
(we have included the variables $z$ and $\zt$ for later use). In the case where $w$ is not necessarily independent of $z$, we obtain that
\begin{align*}
\pt_t w=i(\Delta_x-|w|^2-e)w+i\pt_z^2 w.
\end{align*}
Thus if we can prove that the term $i\pt_z^2 w$ does not contribute to the Morawetz identity, then the above formula continues to hold for all $w$ which are not necessarily independent of $z$. We claim that this is indeed the case. To see this, direct calculation shows that the contribution of $i\pt_z^2 w$ is given by
\begin{align}
-&\,4\int\mathrm{Im}(\bar{w}(t,y,\zt)\pt^2_z w(t,y,\zt))
\frac{x-y}{|x-y|}\cdot\mathrm{Im}(\bar{w}(t,x,z)\nabla_x w(t,x,z))\,dxdydzd\zt
\label{5.17}\\
-&\,2\int|w(t,y,\zt)|^2\frac{x-y}{|x-y|}\cdot\mathrm{Re}(\pt_z^2\bar{w}(t,x,z)\nabla_x w(t,x,z))\,dxdydzd\zt
\label{5.18}\\
+&\,2\int|w(t,y,\zt)|^2\frac{x-y}{|x-y|}\cdot\mathrm{Re}(\bar{w}(t,x,z)\pt_z^2\nabla_x w(t,x,z))\,dxdydzd\zt.\label{5.19}
\end{align}
Since $\bar{w}\pt_z^2 w=\pt_z(\bar{w}\pt_z w)-|\pt_z w|^2$, $\eqref{5.17}=0$ follows from integration by parts, the periodic boundary condition of $w$ along the $z$-direction and the fact that $|\pt_z w|^2$ is real-valued. Another application of integration by parts yields $\eqref{5.18}+\eqref{5.19}=0$, which completes the proof.

We next estimate \eqref{8.3} to \eqref{8.8} term by term. By the nonnegativity of the integrand of the second term in the r.h.s. of \eqref{8.3} we know that $\int_{J}\eqref{8.3}\,dt\gtrsim \|w\|^4_{L_{t,x}^4L_z^2}$. From an application of the Cauchy-Schwarz inequality and the convexity of $x\mapsto |x|$, it was shown in \cite{intermora} that $\eqref{8.5}+\eqref{8.6}\geq 0$. Also, from \cite{intermora} (in conjunction with the integration over $\T_z\times\T_{\zt}$) we know that
\[|I(t)|\lesssim\|w(t)\|^2_{L_{x,z}^2}\|w(t)\|_{\dot{H}_x^{\frac12}L_{z}^2}^2\leq
\|w\|^2_{L_t^\infty L_{x,z}^2}\|w\|_{L_t^\infty \dot{H}_x^{\frac12}L_z^2}^2 .\]
It remains to estimate \eqref{8.7}, \eqref{8.7+} and \eqref{8.8}. Using H\"older's inequality one obtains that
\begin{align*}
\int_J\eqref{8.7}\,dt\lesssim \|ew\|_{L_{t,x,z}^1}\|w\|_{L_{t}^\infty L_{x,z}^2}\|w\|_{L_{t}^\infty \dot{H}_x^{1}L_z^2}.
\end{align*}
The term $e$ contains products in form of $vw_1w_2$, where $w_i\in\{v,w\}$\footnote{We ignore here and below the difference between a function and its complex conjugate as the Lebesgue-type norms do not see this difference.}. Using H\"older and the embeddings $H_z^{\frac12+}\hookrightarrow L_z^\infty$ and $L_y^4\hookrightarrow L_y^2$ we obtain
\begin{align*}
\|ew\|_{L_{t,x,z}^1}&\lesssim \|v\|_{L_t^2 L_{x,z}^\infty}(\|v\|^2_{L_{t,x}^4L_z^2}
+\|w\|^2_{L_{t,x}^4L_z^2})\|w\|_{L_t^\infty L_x^2 L_z^\infty}\nonumber\\
&\lesssim \|v\|_{L_t^2 L_{x,z}^\infty}\|v\|^2_{L_{t,x,z}^4}\| w\|_{L_t^\infty L_x^2 H_z^{\frac12 +}}
+
\|v\|_{L_t^2 L_{x,z}^\infty}\|w\|^2_{L_{t,x}^4L_z^2}\| w\|_{L_t^\infty L_x^2 H_z^{\frac12 +}}
\end{align*}
which in turn implies
\begin{align*}
\int_J\eqref{8.7}\,dt&\lesssim
\|v\|_{L_t^2 L_{x,z}^\infty}\|v\|^2_{L_{t,x,z}^4}
\| w\|_{L_t^\infty L_{x,z}^2}\| w\|_{L_t^\infty L_x^2 H_z^{\frac12+}}\| w\|_{L_t^\infty H_{x,z}^1}\nonumber\\
&+\|v\|_{L_t^2 L_{x,z}^\infty}\|w\|^2_{L_{t,x}^4L_z^2}
\| w\|_{L_t^\infty L_{x,z}^2}\| w\|_{L_t^\infty L_x^2 H_z^{\frac12+}}\| w\|_{L_t^\infty H_{x,z}^1}.
\end{align*}
To estimate \eqref{8.7+}, we first recall the following identity proved in the proof of \cite[Lem. 4.3]{ShenSofferWu21} (with modifications due to the change of the quintic case to the cubic case):
\begin{align*}
&\int\frac{x-y}{|x-y|}\cdot\re((|u|^2u-|w|^2w)\nabla_x \bar w)(t,x,z)\,dxdz\nonumber\\
=&-\frac12\int\frac{1}{|x-y|}(|u|^4-|w|^4)(t,x,z)\,dxdz-\int\frac{x-y}{|x-y|}\cdot\re(|u|^2u\nabla_x \bar v)(t,x,z)\,dxdz\nonumber\\
=&:I_1+I_2.
\end{align*}
For $I_1$, using Hardy's inequality we first obtain
\begin{align*}
\int_J I_1|w(t,y,\zt)|^2\,dyd\zt dt&\lesssim\int_J\bg|\int(|u|^4-|w|^4)(t,x,z)dxdz\bg|
\times\sup_{x}\|\frac{1}{|x-y|^{\frac12}}w(t,y,\zt)\|^2_{L_{y,\zt}^2}\,dt\nonumber\\
&\lesssim\|w\|^2_{L_t^\infty \dot{H}_x^{\frac12}L_z^2}\int_J\bg|\int(|u|^4-|w|^4)(t,x,z)dxdz\bg|\,dt\nonumber\\
&=:\|w\|^2_{L_t^\infty \dot{H}_x^{\frac12}L_z^2}\int_JI_{11}\,dt.
\end{align*}
In view of interpolation, we only need to consider the terms $v^4$ and $vw^3$ appearing in $I_{11}$. For $v^4$, we use H\"older to deduce the estimate
\begin{align*}
\text{Contribution of $v^4$ }\lesssim \|v\|_{L_t^2 L_{x,y}^\infty} \|v\|^3_{L_t^6 L_{x,y}^3}.
\end{align*}
Using H\"older and Sobolev we also infer that
\begin{align*}
\text{Contribution of $vw^3$ }&\lesssim \|v\|_{L_t^2 L_{x,y}^\infty}\|w\|_{L_{t,x}^4 L_z^2}^2\|w\|_{L_t^\infty L_x^2 L_z^\infty}
\nonumber\\
&\lesssim
\|v\|_{L_t^2 L_{x,y}^\infty}\|w\|_{L_{t,x}^4 L_z^2}^2\| w\|_{L_t^\infty L_x^2H_{z}^{\frac12 +}}.
\end{align*}
For $I_2$, bounding $\frac{x-y}{|x-y|}$ by $1$, we need therefore to estimate the terms $v^3\nabla_x v $ and $w^3 \nabla_x v$. Using H\"older and Sobolev we obtain
\begin{align*}
\text{Contribution of $v^3\nabla_x v $ }\lesssim \|\la\nabla\ra v\|_{L_t^2 L_{x,y}^\infty} \|v\|^3_{L_t^6 L_{x,y}^3}
\end{align*}
and
\begin{align*}
\text{Contribution of $w^3\nabla_x v $ }\lesssim \|\la\nabla\ra v\|_{L_t^2 L_{x,z}^\infty}\| w\|^2_{L_{t,x}^4L_z^2}\| w\|_{L_t^\infty L_x^2 H_z^{\frac12 +}}.
\end{align*}
This implies
\begin{align*}
\int_J I_2|w(t,y,\zt)|^2\,dyd\zt dt&\lesssim \|w\|_{L_{t}^\infty L_{x,z}^2}^2
\|\la\nabla\ra v\|_{L_t^2 L_{x,y}^\infty} \|v\|^3_{L_t^6 L_{x,y}^3}\nonumber\\
&+\|w\|_{L_{t}^\infty L_{x,z}^2}^2\|\la\nabla\ra v\|_{L_t^2 L_{x,z}^\infty}\| w\|^2_{L_{t,x}^4L_z^2}\| w\|_{L_t^\infty L_x^2 H_z^{\frac12 +}}.
\end{align*}
Summing up at this point, we have thus proved
\begin{align*}
\int_J\eqref{8.7+}\,dt&\lesssim \|v\|_{L_t^2 L_{x,y}^\infty} \|v\|^3_{L_t^6 L_{x,y}^3}
\|w\|^2_{L_t^\infty H_x^{\frac12}L_z^2}
+\|v\|_{L_t^2 L_{x,y}^\infty}\|w\|_{L_{t,x}^4 L_z^2}^2\| w\|^3_{L_t^\infty L_x^2 H_z^{\frac12+}}
\nonumber\\
&+\|\la\nabla\ra v\|_{L_t^2 L_{x,y}^\infty} \|v\|^3_{L_t^6 L_{x,y}^3}\|w\|_{L_{t}^\infty L_{x,z}^2}^2\nonumber\\
&+\|\la\nabla\ra v\|_{L_t^2 L_{x,z}^\infty}\| w\|^2_{L_{t,x}^4L_z^2}\|w\|_{L_{t}^\infty L_{x,z}^2}^2\|w\|_{L_t^\infty L_x^2 H_z^{\frac12+}}.
\end{align*}
Finally, we can use H\"older and Hardy to estimate \eqref{8.8} similarly:
\begin{align*}
\int_J\eqref{8.8}\,dt&\lesssim
\|v\|_{L_t^2 L_{x,z}^\infty}\|v\|^2_{L_{t,x,z}^4}
\| w\|^3_{L_t^\infty L_x^2 H_z^{\frac12+}}
+\|v\|_{L_t^2 L_{x,z}^\infty}\|w\|^2_{L_{t,x}^4L_z^2}
\|w\|^3_{L_t^\infty L_x^2 H_z^{\frac12+}}.
\end{align*}
The desired proof then follows from the previous calculations and the fundamental theorem of calculus.
\end{proof}

\subsection{Almost conservation laws}
We prove in this subsection that the mass and energy of a solution $w$ of \eqref{nls_w} are almost conserved in the sense that they are not larger than their initial size twice at any time. To be more precise, we define the mass and energy of $w$ by
\[M(w(t))=\|w(t)\|_{L_{x,y}^2}^2,\quad E(w(t))=\frac{1}{2}\|\nabla w(t)\|_{L_{x,y}^2}^2+\frac{1}{4}\|u(t)\|_{L_{x,y}^4}^4\]
respectively. Then we have the following almost conservation laws.
\begin{lemma}[Almost conservation laws]\label{almost conv lem}
Let $s<0$, $A>0$ and $K$ satisfy \eqref{condition on K}. Suppose that $w$ is a solution of \eqref{nls_w} on $[0,T]$ with some $T>0$. Then there exists some dyadic $N_0=N_0(A)$ with the following property: Assume that $\hat{v}_0$ is supported on $\{|\xi|\geq N_0\}$ and
\begin{gather}
\|u_0\|_{H_{x,y}^s}+\|v_0\|_{H_{x,y}^s}+\|v\|_{W(\R)} \leq A,\label{4.30}\\
M(w(0))\leq A N_0^{-2s},\quad E(w(0))\leq A N_0^{2(1-s)}\label{4.31}.
\end{gather}
Then
\begin{gather*}
\sup_{t\in[0,T]}M(w(t))\leq 2A N_0^{-2s},\quad \sup_{t\in[0,T]}E(w(t))\leq 2A N_0^{2(1-s)}.
\end{gather*}
\end{lemma}

\begin{remark}
Because of the embedding $H_{x,y}^{s_1}\hookrightarrow H_{x,y}^{s_2}$ for $s_1>s_2$ we have therefore given Lemma \ref{almost conv lem} only for the (harder) case $s<0$, where even the mass conservation of the NLS is \textit{a priori} unavailable.
\end{remark}

\begin{proof}
We first make the convention that all implicit constants in the upcoming calculations will only depend on $A$. Also write $I=[0,T]$. Let $N_0=N_0(A)$ be a dyadic number chosen later. By a standard continuity argument, it suffices to show that if
\begin{gather}
\sup_{t\in I}M(w(t))\leq 2A N_0^{-2s},\quad \sup_{t\in I}E(w(t))\leq 2A N_0^{2(1-s)},\label{bootstrap1}
\end{gather}
then
\begin{gather*}
\sup_{t\in I}M(w(t))\leq \frac32 A N_0^{-2s},\quad \sup_{t\in I}E(w(t))\leq \frac32A N_0^{2(1-s)}.
\end{gather*}
As usual, we omit the dependence of the norms on $I$ in the calculations. By our assumption, we first know that
\begin{align}\label{12.2}
\|v\|_{L_{t,x,y}^4}+\|v\|_{L_t^6 L_{x,y}^3}\lesssim \|v\|_{W}\lesssim 1.
\end{align}
Since $\hat{v}_0$ is supported on the high frequency region, we also deduce that for $0\leq l<s-\frac13+\frac{K}{2}$
\[\||\nabla|^lv\|_{L_t^2 L_{x,y}^\infty}\lesssim
 N_0^{(l+\frac{1}{3}-\frac{K}{2}-s)+}\|v\|_W\lesssim N_0^{(l+\frac{1}{3}-\frac{K}{2}-s)+}\lesssim 1.\]
Particularly, since $K>\max\{-4s+\frac{20}{3},\frac83\}\geq-2s+\frac{14}{3}$, we can choose $l\in\{0,1,2\}$ to obtain
\begin{align}
\|v\|_{L_t^2 L_{x,y}^\infty}\lesssim N_0^{(\frac{1}{3}-\frac{K}{2}-s)+},\quad
\|\la\nabla \ra v\|_{L_t^2 L_{x,y}^\infty}\lesssim N_0^{(\frac{4}{3}-\frac{K}{2}-s)+},\quad
\|\Delta v\|_{L_t^2 L_{x,y}^\infty}\lesssim N_0^{(\frac{7}{3}-\frac{K}{2}-s)+}.\label{12.3c}
\end{align}
By \eqref{bootstrap1} and interpolation we also infer that for $\alpha\in[0,1]$
\begin{align}
\||\nabla|^s w\|_{L_t^\infty L_{x,y}^2}\lesssim N_0^{\alpha-s}.\label{12.3}
\end{align}
Using Lemma \ref{inter mora lem}, \eqref{12.2}, \eqref{12.3c}, \eqref{12.3} and Young's inequality, we then obtain for arbitrary $\vare>0$ there exists some $C(\vare)>0$ such that
\begin{align*}
\|w\|_{L_{t,x}^4L_y^2}^4&\lesssim N_0^{1-4s}+N_0^{(\frac{11}{6}-\frac{K}{2}-4s)+}
+N_0^{(\frac43-\frac{K}{2}-3s)+}
+N_0^{(\frac{11}{6}-\frac{K}{2}-4s)+}\|w\|_{L_{t,x}^4L_y^2}^2\nonumber\\
&\lesssim N_0^{1-4s}+N_0^{(\frac{11}{6}-\frac{K}{2}-4s)+}
+N_0^{(\frac43-\frac{K}{2}-3s)+}
+C(\vare)N_0^{(\frac{11}{3}-K-8s)+}+\vare\|w\|_{L_{t,x}^4L_y^2}^4.
\end{align*}
Using $K>-4s+\frac{20}{3}>-4s+\frac{8}{3}$ we infer that\footnote{In the case $s<0$ the condition $K>\frac83$ given in \eqref{condition on K} is irrelevant. However, when $s\geq 0$ we need to replace the term $N_0^{1-4s}$ to $N_0$. In this case, the condition $K>\frac83$ will also come into play in order to guarantee \eqref{footnote}.}
\begin{align}\label{footnote}
N_0^{(\frac{11}{6}-\frac{K}{2}-4s)+}+N_0^{(\frac43-\frac{K}{2}-3s)+}+N_0^{(\frac{11}{3}-K-8s)+}\lesssim N_0^{1-4s}.
\end{align}
Hence by choosing $\vare$ suitably small, we deduce that
\begin{align}\label{12.6}
\|w\|_{L_{t,x}^4L_y^2}\lesssim N_0^{\frac{1}{4}-s}.
\end{align}
Next, from the proof of \cite[Lem. 4.5, Prop. 5.4]{ShenSofferWu21} (in conjunction with the periodic boundary conditions in our case) we know that
\begin{align}
\bg|\frac{d}{dt}M(w(t))\bg|&\leq 2\bg|\int (|w+v|^2(w+v)-|w|^2w) \bar w\,dxdz)\bg|,\label{12.7}\\
\bg|\frac{d}{dt}E(w(t))\bg|&\leq \bg|\int (|w+v|^2(w+v)\Delta\bar v\,dxdz)\bg|\label{12.8}.
\end{align}
Using \eqref{12.3c}, \eqref{12.3}, \eqref{12.6}, H\"older and the embeddings $H_y^{\frac12+}\hookrightarrow L_y^\infty$ and $L_y^4\hookrightarrow L_y^2$ we infer that
\begin{align*}
\int_I\text{r.h.s. of \eqref{12.7}}\,dt&\lesssim \|v\|_{L_t^2 L_{x,y}^\infty}(\|v\|^2_{L_{t,x}^4L_y^2}
+\|w\|^2_{L_{t,x}^4L_y^2})\|w\|_{L_t^\infty L_x^2 L_y^\infty}\nonumber\\
&\lesssim \|v\|_{L_t^2 L_{x,y}^\infty}(\|v\|^2_{L_{t,x,y}^4}
+\|w\|^2_{L_{t,x}^4L_y^2})\|w\|_{L_t^\infty L_x^2 H_y^{\frac12+}}
\lesssim N_0^{(\frac43-\frac{K}{2}-4s)+}
\end{align*}
and
\begin{align*}
\int_I\text{r.h.s. of \eqref{12.8}}\,dt&\lesssim \|\Delta v\|_{L_t^2 L_{x,y}^\infty}(\|v\|^2_{L_{t,x}^4L_y^2}
+\|w\|^2_{L_{t,x}^4L_y^2})\|w\|_{L_t^\infty L_x^2 H_y^{\frac12+}}
\lesssim N_0^{(\frac{10}{3}-\frac{K}{2}-4s)+}.
\end{align*}
Using integration by parts and $K>-4s+\frac83$ we deduce that
\begin{align*}
M(w(t))&\leq M(w(0))+\int_I\bg|\frac{d}{dt}M(w(t))\bg|\,dt\leq AN_0^{-2s}+CN_0^{(\frac43-\frac{K}{2}-4s)+}\nonumber\\
&= AN_0^{-2s}+CN_0^{-2s}N_0^{(\frac43-\frac{K}{2}-2s)+}\leq \frac32 AN_0^{-2s}
\end{align*}
by choosing $N_0=N_0(A)$ sufficiently large (which is possible since $K>-4s+\frac{20}{3}>-4s+\frac83$). Similarly we also deduce that $E(w(t))\leq \frac32 AN_0^{2-2s}$, we omit the repeating details here. This completes the desired proof.
\end{proof}

\subsection{Stability theories}\label{sec 4.4}
As a final preliminary, we prove in this subsection a stability result for the equation \eqref{nls_w}.
\begin{lemma}[Short time stability]\label{short time stab lem}
Let $\tilde w\in X^1(I)$ be a solution of \eqref{nls} on a time interval $I$ with $0\in I$ and $\tilde w(0)=\tilde w_0\in H_{x,y}^1$. Suppose that $\|w_0\|_{H_{x,y}^1}\leq L$. Then there exists $\vare_0=\vare_0(L)$ such that if
\begin{align*}
\|\tilde w\|_{Z(I)}\leq \vare_0
\end{align*}
and
\begin{align*}
\|w_0-\tilde w_0\|_{H_{x,y}^1}+\|v\|_{W(I)}\leq \vare<\vare_0,
\end{align*}
then there exists a solution $w\in X^1(I)$ of \eqref{nls_w} such that
\begin{gather*}
\|w-\tilde w\|_{L_t^\infty H_{x,y}^1(I)}+\|w-\tilde w\|_{X^1(I)}\leq C(L)\vare^{\frac14}.
\end{gather*}
\end{lemma}

\begin{proof}
By considering the contraction mapping
\[\Phi_1(\tilde w)=e^{it\Delta}\tilde w_0-i\int_0^t e^{i(t-s)\Delta}(|\tilde w|^2 \tilde w)(s)\,ds\]
in the space
\[S(I):=\{u\in X^1(I):\|u\|_{X^1}\leq 2CL,\,\|u\|_{Z}\leq 2\vare_0\},\]
we may argue as in the proof of Lemma \ref{local theo lem} to find some $\vare_0=\vare_0(L)$ such that $\|\tilde w\|_{X^1}\leq 2CL$. By writing $g=w-\tilde w$ we see that the existence of $w\in X^1(I)$ is equivalent to the existence of $g\in X^1(I)$ solving the perturbed NLS
\begin{align}\label{nls_g1}
(i\pt_t+\Delta)g=|g+v+\tilde w|^2(g+v+\tilde w)-|\tilde w|^2\tilde w=:F(g(t)).
\end{align}
with $g(0)=g_0=w_0-\tilde w_0$. Similarly, we consider the contraction mapping
\[\Phi_2(g)=e^{it\Delta}(w_0-\tilde w_0)-i\int_0^t e^{i(t-s)\Delta}F(g(s))\,ds\]
in the space $S(I)$ to establish the existence of $g$ on $I$ by choosing $\vare_0=\vare_0(L)$ sufficiently small, whence also the existence of $w$ on $I$. In this case we also have
\begin{align*}
\|g\|_{X^1}\leq 2CL,\quad\|g\|_{Z}\leq 2\vare_0,
\end{align*}
whence also $\|g\|_{Z'}\lesssim L^{\frac14}\vare_0^{\frac34}$. Finally, using the nonlinear estimates given in Section \ref{sec 3.2} we infer that
\begin{align}\label{4.49}
\|g\|_{X^1}&\leq C\|w_0-\tilde w_0\|_{H_{x,y}^1}+C(\|g^3\|_{N^1}+\|v^3\|_{N^1}+\|g\tilde w^2\|_{N^1}+\|v\tilde w^2\|_{N^1})\nonumber\\
&\leq C\|w_0-\tilde w_0\|_{H_{x,y}^1}+C\|g\|^2_{Z'}\|g\|_{X^1}+C\|\la\nabla\ra v\|_{L_t^2L_{x,y}^\infty}\|v\|^2_{L_{t,x,y}^4}\nonumber\\
&+C(\|\tilde w\|_{Z'}^2\|g\|_{X^1}+\|\tilde w\|_{Z'}\|g\|_{Z'}\|w\|_{X^1})
+C\|\la\nabla\ra^{s+\kappa} v\|^{\frac12}_{L_t^2 L_{x,y}^\infty}\|\tilde w\|^2_{X^1}\nonumber\\
&\leq C\vare+CL^{\frac12}\vare_0^{\frac32}\|g\|_{X^1}+C\vare^3+CL^{\frac12}\vare_0^{\frac32}\|g\|_{X^1}+CL^{\frac54}\vare_0^{\frac34}\|g\|_{X^1}
+CL^2\vare^{\frac12}.
\end{align}
The desired claim follows from choosing $\vare_0(L)$ sufficiently small and the embedding $X^1\hookrightarrow L_t^\infty H_{x,y}^1$.
\end{proof}

\begin{lemma}[Long time stability]\label{long time stab lem}
Let $\tilde w$ be a solution of \eqref{nls} on a time interval $I$ with $0\in I$ and $\tilde w(0)=\tilde w_0\in H_{x,y}^1$. Suppose that
\begin{align*}
\|\tilde w\|_{L_t^\infty H_{x,y}^1(I)}+\|\tilde w\|_{Z (I)}\leq L
\end{align*}
Then there exists some $\beta=\beta(L,\vare_0)\ll 1$, where $\vare_0$ is the number defined in Lemma \ref{short time stab lem}, such that if
\begin{align*}
\|v\|_{W(I)}\leq \beta,
\end{align*}
then \eqref{nls_w} possesses a solution $w$ on $I$ with $w(0)=\tilde w_0$. Moreover, there exists some $\alpha=\alpha(L,\vare_0)\in(0,1)$ such that
\begin{align*}
\|w-\tilde w\|_{X^1(I)}\leq C(L,\vare_0)\beta^\alpha.
\end{align*}
\end{lemma}

\begin{proof}
First, we partition $I=\cup_{j=1}^J I_j=:\cup_{j=1}^J [t_{j-1},t_j)$ with $J=J(L,\vare_0)\in\N$ such that
$$c_1^{-1}\vare_0\leq \|\tilde w\|_{Z(I_j)}\leq\vare_0$$
for all $j=1,\cdots, J$, where $c_1>1$ is some positive constant. Let $C(L)$ be the number defined in Lemma \ref{short time stab lem}. We then let $\beta$ be a positive number such that
\begin{align*}
\beta\leq C(L)^{\frac43}\bg(\frac{\vare_0}{C(L)^{\frac43}}\bg)^{4^{J-1}}.
\end{align*}
With this choice of $\beta$ it is easy to verify that if $a_1:=\beta$ and $a_{j+1}:=C(L)a_j^{\frac14}$, then \[a_j=C(L)^{\frac43(1-\frac{1}{4^{j-1}})}\beta^{\frac{1}{4^{j-1}}}\]
and $a_j\leq \vare_0$ for all $j=1,\cdots,J$. Hence we may apply Lemma \ref{short time stab lem} on all $I_j$ to infer that
\begin{align*}
\|w-\tilde w\|_{L_t^\infty H_{x,y}^1(I_j)}+\|w-\tilde w\|_{X^1(I_j)}\leq a_j\leq \vare_0,
\end{align*}
from which we also deduce that
\begin{align*}
\|w-\tilde w\|_{X^1(I)}\leq \sum_{j=1}^J C(L)^{\frac43(1-\frac{1}{4^{j-1}})}\beta^{\frac{1}{4^{j-1}}}\leq JC(L)^{\frac43}\beta^{\frac{1}{4^{J-1}}},
\end{align*}
as desired.
\end{proof}

\subsection{Conclusion}\label{sec 4.5}
We complete in this subsection the proof of Theorem \ref{main thm}. For the final proof, we also need the following large data scattering result for the defocusing cubic NLS on $\R^3\times\T$.

\begin{theorem}[Large data scattering of \eqref{nls}, \cite{RmT1}]\label{thm 4.7}
Let $u_0\in H_{x,y}^1$ and let $u\in X^1$ be a local solution of \eqref{nls} defined on a neighborhood of $0$ with $u(0)=u_0$. Then $u$ is a global scattering solution. In particular, we have $\|u\|_{Z(\R)}\leq C(\|u_0\|_{H_{x,y}^1})$.
\end{theorem}

Having all the preliminaries we are in a position to prove Theorem \ref{main thm}.

\begin{proof}[Proof of Theorem \ref{main thm}]
It suffices to consider the (harder) case $s<0$. By Lemma \ref{finiteness lemma} we know that for a.e. $\omega\in \Omega$ we have
\begin{align}\label{4.53}
\|f^\omega\|_{H_{x,y}^s}+\|f^\omega\|_{L_{x,y}^4}+\|e^{it\Delta}f^\omega\|_{W(\R)}\leq C_\omega<\infty.
\end{align}
Now set
\[u_0=f^\omega,\quad v_0=P_{> N_0}f^\omega,\quad w_0= P_{\leq N_0}f^\omega,\quad v=e^{it\Delta}v_0.\]
in order to apply the results stated in Section \ref{sec 4.1} to Section \ref{sec 4.4}, where $N_0$ is some to be determined dyadic number. Using Bernstein and \eqref{4.53} we obtain
\begin{align*}
\|w_0\|_{L_{x,y}^2}&\lesssim N_0^{-s}\|f^\omega\|_{H_{x,y}^s}\lesssim C_\omega N_0^{-s},\\
\|\nabla w_0\|_{L_{x,y}^2}&\lesssim N_0^{1-s}\|f^\omega\|_{H_{x,y}^s}\lesssim C_\omega N_0^{1-s},\\
\|u_0\|_{L_{x,y}^4}&= \|f^\omega\|_{L_{x,y}^4}\lesssim C_\omega.
\end{align*}
By setting $A:=C(1+C_\omega)^4$ we see that \eqref{4.30} and \eqref{4.31} are satisfied. Let $N_0=N_0(A)$ be anchored according to Lemma \ref{almost conv lem}. Now we set in Lemma \ref{long time stab lem} the number $L$ as $L=2CA N_0^{2-2s}$. Let $\beta$ in Lemma \ref{long time stab lem} be determined according to $L$. We then partition $[0,\infty)=\cup_{j=1}^J I_j$ with $I_j=[t_{j-1},t_j)$ such that
\[c_1^{-1}\beta\leq \|v\|_{W(I_j)}\leq \beta\]
for all $j=1,\cdots,J$, where $c_1>1$ is some positive constant. Let $\tilde w=\tilde w^j$ be the solution of \eqref{nls} with $\tilde w^j(t_{j-1})=w(t_{j-1})$. By the inductive hypothesis we may apply Lemma \ref{long time stab lem} to infer that
$$\|w-\tilde w^j\|_{X^1(I_j)}\leq C(L)\beta^\alpha.$$
From Theorem \ref{thm 4.7} and our inductive hypothesis we also know that
$$\|\tilde w^j\|_{X^1(I_j)}\leq C(\|w(t_{j-1})\|_{H_{x,y}^1})=C(L).$$
On the other hand, using Lemma \ref{almost conv lem} we have $\|w(t_{j})\|_{H_{x,y}^1}\leq 2AN_0^{2-2s}$, thus the previous arguments are applicable for the interval $I_{j+1}$. Applying therefore the previous arguments inductively for all $I_j$, $j=1,\cdots,J$, with $\tilde w^{j+1}(t_j)=w(t_j)$ and followed by summing up the sub-estimates, we deduce that $\|w\|_{X^1(0,\infty)}<\infty$. A similar result holds also for $(-\infty,0)$, thus $\|w\|_{X^1(\R)}<\infty$, from which the global well-posedness of $w$ and consequently $u=w+v$ follows. The scattering of $u$ then follows from $\|w\|_{X^1(\R)}+\|v\|_{W(\R)}<\infty$ and a standard application of the Strichartz estimates (which are similar to the ones applied in the proof of Lemma \ref{local theo lem}, see also \cite{HaniPausader,RmT1} for similar arguments), we thus omit the details here. The proof of Theorem \ref{main thm} is therefore complete.
\end{proof}

%%%%%%%%%%%%%%%%%%%%%%%%%%%%%
%\addcontentsline{toc}{section}{References}
%\bibliography{biharmonic}

\begin{thebibliography}{10}

\bibitem{Benyi2015}
{\sc B{\'e}nyi, {\'A}., Oh, T., and Pocovnicu, O.}
\newblock On the probabilistic {Cauchy} theory of the cubic nonlinear
  {Schr{\"o}dinger} equation on {{\(\mathbb {R}^d\)}}, {{\(d \geq 3\)}}.
\newblock {\em Trans. Am. Math. Soc., Ser. B 2\/} (2015), 1--50.

\bibitem{Benyi2019}
{\sc B{\'e}nyi, {\'A}., Oh, T., and Pocovnicu, O.}
\newblock Higher order expansions for the probabilistic local {Cauchy} theory
  of the cubic nonlinear {Schr{\"o}dinger} equation on {{\(\mathbb {R}^3\)}}.
\newblock {\em Trans. Am. Math. Soc., Ser. B 6\/} (2019), 114--160.

\bibitem{benyi_oh_pocov_survey}
{\sc B{\'e}nyi, {\'A}., Oh, T., and Pocovnicu, O.}
\newblock On the probabilistic {Cauchy} theory for nonlinear dispersive {PDEs}.
\newblock In {\em Landscapes of time-frequency analysis. Based on talks given
  at the inaugural conference on aspects of time-frequency analysis, Turin,
  Italy, July 5--7, 2018}. Cham: Birkh{\"a}user, 2019, pp.~1--32.

\bibitem{Bourgain2}
{\sc Bourgain, J.}
\newblock Fourier transform restriction phenomena for certain lattice subsets
  and applications to nonlinear evolution equations. {II}. {T}he
  {K}d{V}-equation.
\newblock {\em Geom. Funct. Anal. 3}, 3 (1993), 209--262.

\bibitem{BourgainProb1}
{\sc Bourgain, J.}
\newblock Periodic nonlinear {Schr{\"o}dinger} equation and invariant measures.
\newblock {\em Commun. Math. Phys. 166}, 1 (1994), 1--26.

\bibitem{BourgainProb2}
{\sc Bourgain, J.}
\newblock Invariant measures for the 2d-defocusing nonlinear {Schr{\"o}dinger}
  equation.
\newblock {\em Commun. Math. Phys. 176}, 2 (1996), 421--445.

\bibitem{l2decoupling}
{\sc Bourgain, J., and Demeter, C.}
\newblock The proof of the {$l^2$} decoupling conjecture.
\newblock {\em Ann. of Math. (2) 182}, 1 (2015), 351--389.

\bibitem{Brereton}
{\sc Brereton, J.}
\newblock Almost sure local well-posedness for the supercritical quintic {NLS}.
\newblock {\em Tunis. J. Math. 1}, 3 (2019), 427--453.

\bibitem{BringmannRadNLW2020}
{\sc Bringmann, B.}
\newblock Almost-sure scattering for the radial energy-critical nonlinear wave
  equation in three dimensions.
\newblock {\em Anal. PDE 13}, 4 (2020), 1011--1050.

\bibitem{Bringmann2021NLW}
{\sc Bringmann, B.}
\newblock Almost sure scattering for the energy critical nonlinear wave
  equation.
\newblock {\em Am. J. Math. 143}, 6 (2021), 1931--1982.

\bibitem{BurqTzvetkov1}
{\sc Burq, N., and Tzvetkov, N.}
\newblock Random data {Cauchy} theory for supercritical wave equations {I}:
  {Local} theory.
\newblock {\em Invent. Math. 173}, 3 (2008), 449--475.

\bibitem{BurqTzvetkov2}
{\sc Burq, N., and Tzvetkov, N.}
\newblock Random data {Cauchy} theory for supercritical wave equations. {II}.
  {A} global existence result.
\newblock {\em Invent. Math. 173}, 3 (2008), 477--496.

\bibitem{Camps}
{\sc Camps, N.}
\newblock Scattering for the cubic {Schr{\"o}dinger} equation in 3d with
  randomized radial initial data.
\newblock {\em Trans. Am. Math. Soc. 376}, 1 (2023), 285--333.

\bibitem{CubicR2T1Scattering}
{\sc Cheng, X., Guo, Z., Yang, K., and Zhao, L.}
\newblock On scattering for the cubic defocusing nonlinear {S}chr\"{o}dinger
  equation on the waveguide {$\Bbb R^2 \times \Bbb T$}.
\newblock {\em Rev. Mat. Iberoam. 36}, 4 (2020), 985--1011.

\bibitem{R1T1Scattering}
{\sc Cheng, X., Guo, Z., and Zhao, Z.}
\newblock On scattering for the defocusing quintic nonlinear {S}chr\"{o}dinger
  equation on the two-dimensional cylinder.
\newblock {\em SIAM J. Math. Anal. 52}, 5 (2020), 4185--4237.

\bibitem{Cheng_JMAA}
{\sc Cheng, X., Zhao, Z., and Zheng, J.}
\newblock Well-posedness for energy-critical nonlinear {S}chr\"{o}dinger
  equation on waveguide manifold.
\newblock {\em J. Math. Anal. Appl. 494}, 2 (2021), Paper No. 124654, 14.

\bibitem{ill_posed}
{\sc Christ, M., Colliander, J., and Tao, T.}
\newblock Ill-posedness for nonlinear {S}chrodinger and wave equations, 2003.

\bibitem{intermora}
{\sc Colliander, J.~E., Keel, M., Staffilani, G., Takaoka, H., and Tao, T.~C.}
\newblock Global existence and scattering for rough solutions of a nonlinear
  {Schr{\"o}dinger} equation on {{\(\mathbb R^3\)}}.
\newblock {\em Commun. Pure Appl. Math. 57}, 8 (2004), 987--1014.

\bibitem{DLM19}
{\sc Dodson, B., L{\"u}hrmann, J., and Mendelson, D.}
\newblock Almost sure local well-posedness and scattering for the 4d cubic
  nonlinear {Schr{\"o}dinger} equation.
\newblock {\em Adv. Math. 347\/} (2019), 619--676.

\bibitem{DLM20}
{\sc Dodson, B., L{\"u}hrmann, J., and Mendelson, D.}
\newblock Almost sure scattering for the 4d energy-critical defocusing
  nonlinear wave equation with radial data.
\newblock {\em Am. J. Math. 142}, 2 (2020), 475--504.

\bibitem{HadacHerrKoch2009}
{\sc Hadac, M., Herr, S., and Koch, H.}
\newblock Well-posedness and scattering for the {KP}-{II} equation in a
  critical space.
\newblock {\em Ann. Inst. H. Poincar\'{e} Anal. Non Lin\'{e}aire 26}, 3 (2009),
  917--941.

\bibitem{HaniPausader}
{\sc Hani, Z., and Pausader, B.}
\newblock On scattering for the quintic defocusing nonlinear {S}chr\"{o}dinger
  equation on {$\Bbb R\times\Bbb T^2$}.
\newblock {\em Comm. Pure Appl. Math. 67}, 9 (2014), 1466--1542.

\bibitem{HerrTataruTz1}
{\sc Herr, S., Tataru, D., and Tzvetkov, N.}
\newblock Global well-posedness of the energy-critical nonlinear
  {S}chr\"{o}dinger equation with small initial data in {$H^1(\Bbb T^3)$}.
\newblock {\em Duke Math. J. 159}, 2 (2011), 329--349.

\bibitem{HerrTataruTz2}
{\sc Herr, S., Tataru, D., and Tzvetkov, N.}
\newblock Strichartz estimates for partially periodic solutions to
  {S}chr\"{o}dinger equations in {$4d$} and applications.
\newblock {\em J. Reine Angew. Math. 690\/} (2014), 65--78.

\bibitem{Ionescu2}
{\sc Ionescu, A.~D., and Pausader, B.}
\newblock Global well-posedness of the energy-critical defocusing {NLS} on
  {$\Bbb R\times \Bbb T^3$}.
\newblock {\em Comm. Math. Phys. 312}, 3 (2012), 781--831.

\bibitem{waveguide_ref_3}
{\sc Kengne, E., Vaillancourt, R., and Malomed, B.~A.}
\newblock Bose{\textendash}einstein condensates in optical lattices: the
  cubic{\textendash}quintic nonlinear schrödinger equation with a periodic
  potential.
\newblock {\em Journal of Physics B: Atomic, Molecular and Optical Physics 41},
  20 (2008), 205202.

\bibitem{KenigMerle2006}
{\sc Kenig, C.~E., and Merle, F.}
\newblock Global well-posedness, scattering and blow-up for the
  energy-critical, focusing, non-linear {S}chr\"{o}dinger equation in the
  radial case.
\newblock {\em Invent. Math. 166}, 3 (2006), 645--675.

\bibitem{KMV4dprob}
{\sc Killip, R., Murphy, J., and Visan, M.}
\newblock Almost sure scattering for the energy-critical {NLS} with radial data
  below {{\(H^1(\mathbb{R}^4)\)}}.
\newblock {\em Commun. Partial Differ. Equations 44}, 1 (2019), 51--71.

\bibitem{Luo_Waveguide_MassCritical}
{\sc Luo, Y.}
\newblock Large data global well-posedness and scattering for the focusing
  cubic nonlinear {S}chr\"odinger equation on $\mathbb{R}^2\times\mathbb{T}$,
  2022.

\bibitem{Luo_inter}
{\sc Luo, Y.}
\newblock Normalized ground states and threshold scattering for focusing {NLS}
  on $\mathbb{R}^d\times\mathbb{T}$ via semivirial-free geometry, 2022.

\bibitem{Luo_energy_crit}
{\sc Luo, Y.}
\newblock On long time behavior of the focusing energy-critical {NLS} on
  $\mathbb{R}^d\times\mathbb{T}$ via semivirial-vanishing geometry, 2022.

\bibitem{luo2022sharp}
{\sc Luo, Y.}
\newblock Sharp scattering for focusing intercritical nls on high-dimensional
  waveguide manifolds, 2022.

\bibitem{GWP5d6d}
{\sc Oh, T., Okamoto, M., and Pocovnicu, O.}
\newblock On the probabilistic well-posedness of the nonlinear
  {Schr{\"o}dinger} equations with non-algebraic nonlinearities.
\newblock {\em Discrete Contin. Dyn. Syst. 39}, 6 (2019), 3479--3520.

\bibitem{OhPocovnicu2016NLW}
{\sc Oh, T., and Pocovnicu, O.}
\newblock Probabilistic global well-posedness of the energy-critical defocusing
  quintic nonlinear wave equation on {{\(\mathbb R^3\)}}.
\newblock {\em J. Math. Pures Appl. (9) 105}, 3 (2016), 342--366.

\bibitem{PocovnicuNLW2017}
{\sc Pocovnicu, O.}
\newblock Almost sure global well-posedness for the energy-critical defocusing
  nonlinear wave equation on {{\(\mathbb R^d\)}}, {{\(d=4\)}} and {{\(5\)}}.
\newblock {\em J. Eur. Math. Soc. (JEMS) 19}, 8 (2017), 2521--2575.

\bibitem{waveguide_ref_1}
{\sc Schneider, T.}
\newblock {\em Nonlinear Optics in Telecommunications}.
\newblock Springer Science \& Business Media, Berlin Heidelberg, 2013.

\bibitem{ShenSofferWu21}
{\sc Shen, J., Soffer, A., and Wu, Y.}
\newblock Almost sure scattering for the nonradial energy-critical nls with
  arbitrary regularity in 3d and 4d cases, 2021.

\bibitem{Shen2022}
{\sc Shen, J., Soffer, A., and Wu, Y.}
\newblock Almost sure well-posedness and scattering of the 3d cubic nonlinear
  {S}chr\"{o}dinger equation.
\newblock {\em Comm. Math. Phys. 397}, 2 (Sept. 2022), 547--605.

\bibitem{waveguide_ref_2}
{\sc Snyder, A., and Love, J.}
\newblock {\em Optical Waveguide Theory}.
\newblock Springer Science \& Business Media, Berlin Heidelberg, 2012.

\bibitem{SpitzNLW}
{\sc Spitz, M.}
\newblock On the almost sure scattering for the energy-critical cubic wave
  equation with supercritical data.
\newblock {\em Commun. Pure Appl. Anal. 21}, 12 (2022), 4041--4070.

\bibitem{SpitzNLS}
{\sc Spitz, M.}
\newblock Almost sure local wellposedness and scattering for the
  energy-critical cubic nonlinear {Schr{\"o}dinger} equation with supercritical
  data.
\newblock {\em Nonlinear Anal., Theory Methods Appl., Ser. A, Theory Methods
  229\/} (2023), 33.
\newblock Id/No 113204.

\bibitem{TTVproduct2014}
{\sc Terracini, S., Tzvetkov, N., and Visciglia, N.}
\newblock The nonlinear {S}chr\"{o}dinger equation ground states on product
  spaces.
\newblock {\em Anal. PDE 7}, 1 (2014), 73--96.

\bibitem{Tzvetkov10}
{\sc Tzvetkov, N.}
\newblock Construction of a {Gibbs} measure associated to the periodic
  {Benjamin}-{Ono} equation.
\newblock {\em Probab. Theory Relat. Fields 146}, 3-4 (2010), 481--514.

\bibitem{TNCommPDE}
{\sc Tzvetkov, N., and Visciglia, N.}
\newblock Small data scattering for the nonlinear {S}chr\"{o}dinger equation on
  product spaces.
\newblock {\em Comm. Partial Differential Equations 37}, 1 (2012), 125--135.

\bibitem{TzvetkovVisciglia2016}
{\sc Tzvetkov, N., and Visciglia, N.}
\newblock Well-posedness and scattering for nonlinear {S}chr\"{o}dinger
  equations on {$\Bbb{R}^d\times\Bbb{T}$} in the energy space.
\newblock {\em Rev. Mat. Iberoam. 32}, 4 (2016), 1163--1188.

\bibitem{YuYueZhao2021}
{\sc Yu, X., Yue, H., and Zhao, Z.}
\newblock Global {W}ell-posedness for the focusing cubic {NLS} on the product
  space {$\Bbb{R} \times \Bbb{T}^3$}.
\newblock {\em SIAM J. Math. Anal. 53}, 2 (2021), 2243--2274.

\bibitem{RmT1}
{\sc Zhao, Z.}
\newblock On scattering for the defocusing nonlinear {S}chr\"{o}dinger equation
  on waveguide {$\Bbb R^m\times \Bbb T$} (when {$m = 2,3$}).
\newblock {\em J. Differential Equations 275\/} (2021), 598--637.

\bibitem{ZhaoZheng2021}
{\sc Zhao, Z., and Zheng, J.}
\newblock Long time dynamics for defocusing cubic nonlinear {S}chr\"{o}dinger
  equations on three dimensional product space.
\newblock {\em SIAM J. Math. Anal. 53}, 3 (2021), 3644--3660.

\end{thebibliography}
%\bibliographystyle{acm}

\end{document}